\documentclass{au}
 \presentedby{\dots}
 \received{\dots}{\dots}
\usepackage{amsmath,amssymb,url}
\usepackage{mathtools}  
\usepackage{enumerate}
\usepackage[mathscr]{eucal}
\usepackage{amssymb,bm}
\usepackage{cmll} 
\usepackage{url}
\usepackage{pgf}
\usepackage{tikz}
\usetikzlibrary{arrows,calc,positioning,matrix}
\tikzset{%
 element/.style={draw, shape=circle, fill=white, inner sep=1.5pt},
 order/.style={very thin},
 inclusion/.style={very thin},
 arrow/.style={->, thin, >=to},
 map/.style={->, shorten >=5pt, shorten <=5pt, >=stealth'},
 unary/.style={->, shorten >=2pt, shorten <=2pt, >=stealth'},
 loopy/.style={->, shorten >=2pt, shorten <=2pt, >=stealth, min distance=15pt},
 auto}

\numberwithin{equation}{section}

\theoremstyle{plain}
\newtheorem{theorem}{Theorem}[section]
\newtheorem{lemma}[theorem]{Lemma}

\newtheorem{corollary}[theorem]{Corollary}

\theoremstyle{definition}

\newtheorem{assume}[theorem]{Basic Assumptions}
\newtheorem{remark}[theorem]{Remark}
\newtheorem{example}[theorem]{Example}

\newenvironment{newlist} 
   {\begin{list}{}{\setlength{\labelsep}{0.25cm}
                   \setlength{\labelwidth}{0.65cm}
                      \setlength{\leftmargin}{0.9cm}}}
   {\end{list}}

\newenvironment{Dlist} 
   {\begin{list}{}{\setlength{\labelsep}{0.1cm}
                   \setlength{\labelwidth}{0.85cm} 
                   \setlength{\leftmargin}{0.90cm}}} 
   {\end{list}}

\newenvironment{rmlist} 
   {\begin{list}{}{\setlength{\labelsep}{0.15cm}
                   \setlength{\labelwidth}{0.55cm}
                   \setlength{\leftmargin}{0.65cm}}}
   {\end{list}}

\newenvironment{iiilist} 
   {\begin{list}{}{\setlength{\labelsep}{0.15cm}
                   \setlength{\labelwidth}{0.62cm}
                   \setlength{\leftmargin}{0.66cm}}}
   {\end{list}}

\newenvironment{alist} 
   {\begin{list}{}{\setlength{\labelsep}{0.2cm}
                   \setlength{\labelwidth}{0.45cm}
                   \setlength{\leftmargin}{0.55cm}}}
   {\end{list}}

\newenvironment{Ilist} 
   {\begin{list}{}{\setlength{\labelsep}{0.25cm}
                   \setlength{\labelwidth}{0.65cm}
                   \setlength{\leftmargin}{0.9cm}}}
   {\end{list}}

\newcommand{\cat}[1]{\boldsymbol{\mathscr{#1}}}
\newcommand{\CA}{{\cat A}}
\newcommand{\CB}{\cat B}
\newcommand{\CY}{\cat Y}
\newcommand{\CP}{\cat P}
\newcommand{\CCD}{\cat D}
\newcommand{\CQ}{\cat Q}

\newcommand{\D}[1]{\mathrm D(#1)}
\newcommand{\E}[1]{\mathrm E(#1)}
\newcommand{\ED}[1]{\mathrm {ED}(#1)}
\newcommand{\DE}[1]{\mathrm {DE}(#1)}
\newcommand{\CX}{\cat X}
\newcommand{\CS}{\cat S}
\newcommand{\CO}{\cat O}

\newcommand{\sub}[1]{_{_{\kern-.9pt{\scriptstyle #1}}}}
\newcommand{\medsub}[2]{#1\lower0.6ex\hbox{$\scriptstyle{#2}$}}
\newcommand{\eA}[1]{\medsub e {\kern-0.75pt\A\kern-0.75pt}(#1)}
\newcommand{\esub}[1]{\medsub e {\kern-0.75pt #1 \kern-0.75pt}}
\newcommand{\epsub}[1]{\medsub \varepsilon {\kern-1.25pt #1}}
\newcommand{\esubA}{\medsub e {\kern-0.75pt\A\kern-0.75pt}}
\newcommand{\twiddle}[1]{{\smash{\underset{\raise.4ex\hbox{$\smash\sim$}}
                       {\mathbf{#1}}}\vphantom{\underline{\mathbf{#1}}}}}
\newcommand{\stwiddle}[1]{{\smash{\underset{\raise.3ex\hbox{$\scriptstyle\smash\sim$}}{\mathbf{#1}}}\vphantom{\underline{\underline{\mathbf{#1}}}}}}
\newcommand{\MT}{\twiddle M}
\newcommand{\NT}{\twiddle N}
\newcommand{\ST}{\twiddle S}
\newcommand{\WT}{\twiddle W} 
\newcommand{\DT}{\twiddle D}
\newcommand{\twoBf}{\mathbf 2}
\newcommand{\twoT}{\twiddle 2}

\newcommand{\A}{\mathbf A}
\newcommand{\B}{\mathbf B} 
\newcommand{\C}{\mathbf C}
\newcommand{\N}{\mathbf N}
\newcommand{\Dalg}{\mathbf D}

\newcommand{\M}{\mathbf M}
\renewcommand{\S}{\mathbf S}
\newcommand{\V}{\mathbf V} 
\newcommand{\W}{\mathbf W} 
\newcommand{\X}{\mathbf X}
\newcommand{\Y}{\mathbf Y}
\newcommand{\Tp}{\mathscr{T}}

\DeclareMathOperator{\Clo}{Clo}
\DeclareMathOperator{\End}{End}
\DeclareMathOperator{\graph}{graph}
\newcommand{\Omegamax}[2]{\mathop{\mathrm{max}_{#1}}\Omega^{-1}\hspace*{-1.0pt}(#2)}
\newcommand{\omegamax}[2]{\mathop{\mathrm{max}_{#1}}\{\omega\}^{-1}\hspace*{-1.0pt}(#2)}
\DeclareMathOperator{\dom}{dom}
\DeclareMathOperator{\id}{id}

\newcommand{\ISP}{\mathbb{ISP}}
\newcommand{\IScP}{{\mathbb{IS} _{\mathrm{c}} \mathbb{P}^+}}

\newcommand{\lee}{\le}  
\renewcommand{\leq}{\leqslant}
\renewcommand{\geq}{\geqslant}
\newcommand{\rest}[1]{{\upharpoonright}_{#1}}
\newcommand{\bigcupPhi}{\bigcup_{\omega\in\Omega} \Phi_\omega^\A(\CA(\A, \M))} 
\newcommand{\bigcupPsi}{\bigcup_{\omega\in\Omega} \Psi_\omega^\X(\CX(\X, \MT))}

\begin{document}


\title[Piggyback dualities revisited]{Piggyback dualities revisited}

\author[B. A. Davey]{B. A. Davey} 
\email{B.Davey@latrobe.edu.au}
\address{Department of Mathematics and Statistics, La Trobe University, Victoria 3086, Australia}

\author[M. Haviar]{M. Haviar}
\email{miroslav.haviar@umb.sk}
\address{Faculty of Natural Sciences, Matej Bel University, Tajovsk\'eho 40, 974 01 Bansk\'{a} Bystrica, Slovak Republic}

\author[H. A. Priestley]{H. A. Priestley}
\email{hap@maths.ox.ac.uk}
\address{Mathematical Institute, University of Oxford, Radcliffe Observatory Quarter, Oxford OX2 6GG, United Kingdom}

\thanks{The second  author acknowledges support from Slovak grants VEGA 1/0212/13 and APVV-0223-10.}

\dedicatory{Dedicated to the memory of Ervin Fried and Jiri Sichler}


\begin{abstract}
In natural duality theory, the piggybacking technique is a valuable tool
for constructing dualities. As originally devised by Davey and Werner, and extended  by Davey and Priestley, it can be applied to finitely generated quasivarieties of algebras having term-reducts in a quasivariety for which a well-behaved natural duality is already available. This paper presents a comprehensive study  of the method in a much wider setting: piggyback duality theorems are obtained for suitable prevarieties of structures.  For the first time, and within this extended framework, piggybacking is used to derive  theorems giving criteria for establishing strong dualities and two-for-one  dualities. The general theorems specialise in particular to the familiar situation in which we piggyback on Priestley duality for distributive lattices or Hofmann--Mislove--Stralka duality for semilattices, and
many well-known dualities are thereby subsumed. A selection of new dualities is also presented.
\end{abstract}

\maketitle


\section{Introduction}\label{sec:intro}


This paper gives a systematic, general  treatment  of the method of piggybacking in the context of natural dualities for structures.  The principal results
are Theorems~\ref{thm:pig-general} and~\ref{thm:copig-general}.  These subsume and extend previous uses of the piggybacking technique. For the first time, piggybacking is used to derive strong dualities, albeit under stringent  conditions; see Theorems~\ref{thm:pigstrong1} and~\ref{thm:two4one}.

In broad terms, duality theory seeks to use one category, $\CX$, to reason
about another, $\CA$, with the categories linked by a dual adjunction
or,  better,  by a dual equivalence.
The more tightly the two categories are linked, the more powerful this general strategy will be. Specifically, assume we have contravariant functors $\mathrm{D} \colon \CA \to \CX$ and $\mathrm {E} \colon \CX\to \CA$ and that there exist a unit and counit $e$ and $\varepsilon$ so that $\langle\mathrm{D},\mathrm{E},e, \varepsilon\rangle$ is a dual adjunction.
   An aspect of duality theory that has proved particularly
fruitful is that in which  $\mathrm{D} $ and $\mathrm{E}$ are hom-functors, with
$\mathrm{D} = \CA(-,\M)$ and $\mathrm{E} = \CX(-,\MT)$, where $\M \in \CA$ and
$\MT \in \CX$ are objects with the same underlying set $M$.
Within this very general categorical framework, duality theory
as a tool for algebra has a special niche.  The theory of natural dualities,
in which $\M$ is taken to be an algebra, usually finite, has been developed
to a high level, as was already evident in the 1998 text by Clark and Davey \cite{NDftWA}.  Subsequent advances have featured dualisability
(as witnessed by the monograph \cite{PD05}) and the theory of full dualities and strong dualities (\cite[Chapter~3]{NDftWA} and~\cite{DPW11}). We shall strike out in a different direction. As we have just indicated, the existing core theory of natural dualities focuses on dual representations for \emph{algebras}, specifically algebras in a finitely generated quasivariety. But, while this restricts $\CA$ to be drawn from a very important class of categories,  it is possible to  encompass structures more general than algebras and in certain circumstances to remove  the requirement of finite generation.

The study of natural dualities for finitely generated quasivarieties of \emph{structures}  was initiated by Davey \cite{D06} (and Hofmann~\cite{H02} had earlier considered the  more general setting of finitary limit sketches).
In the present  paper we fill in more of the overall  duality picture in the setting of structures. In  a not unrelated development, the present authors considered situations in which the topology could be moved from one side of a dual adjunction to the other, thereby creating pairs of dualities in partnership~\cite{DHP12}; our Theorem~\ref{thm:two4one} can be seen as  a piggyback-based topology-swapping theorem.  Motivation for consideration of paired dualities came in part
from investigation of  canonical extensions for lattice-based algebras (see \cite{DHP12,DGHP}). More recently, we have shown how these ideas fit into the  broader framework of free constructions which can be viewed as zero-dimensional Bohr compactifications of structures \cite{DHP}.

 The authors of \cite{NDftWA} took a deliberate decision to restrict their treatment to finitely generated quasivarieties (of algebras).
It was already recognised in \cite{DW83} that finite generation is not a necessary condition for a natural duality to exist but, 30 years on, little general theory has been developed and non-finitely generated examples remain tantalisingly scarce: abelian groups (Pontryagin \cite{Pon34}); Ockham algebras \cite{G83,DW85,DW83PB}; certain semilattice-based algebras~\cite[Section~8]{DJPT07}. In  this paper  we  operate within a framework in which we do not restrict to generating structures which are finite.  As a consequence,  topological conditions arise which are absent in the finitely generated  setting. Our main results, as presented in Section~\ref{sec:results}, require the structures under consideration to be total, that is, to contain no partial operations. Nevertheless, the general setting of infinite generating structures is a new departure,  and we shall allow for partial operations where there is no reason to exclude them.

 Most   importantly, we must comment on the role of Davey and Werner's piggybacking technique  \cite{DW85,DW83PB} and its subsequent evolution.  The fundamental idea is very simple. Consider a prevariety $\CB= \ISP(\N)$ for which a well-behaved
full duality is already available---prototypical examples would be~$\CCD$, the variety of bounded distributive lattices, with Priestley duality, and~$\CS$, the variety of unital  meet semilattices, with Hofmann--Mislove--Stralka duality.
Take an algebra $\M$ generating a prevariety $\CA =\ISP(\M)$ for which a
duality is sought.  Assume that $\M$ has a term-reduct $\M^\flat$ in
$\CB$. Seek to hitch a piggyback ride: use this known duality for $\CB$ to build the required duality for $\CA$, using a carrier map $\omega \in \CB(\M^\flat,\N)$
to link the categories involved. As Davey and Werner showed, this idea allowed the  construction of economical dualities, in situations where the existence of a duality was not in question but where general theory (in particular the NU Duality Theorem, where applicable) supplied dual structures too unwieldy to be of practical use.
They also showed how the method had the potential to establish dualisability
for certain non-finitely generated varieties, as witnessed by their treatment of
the variety of Ockham algebras.  Later, Davey and Priestley \cite{DP87} enlarged
the scope of the method  by  allowing a set of carrier maps  in cases where
a single map~$\omega$ does not provide tight enough linkage between~$\CA$ and~$\CB$
for  Davey and Werner's single-carrier piggyback theorem to apply.  More details
on the general method can be found in Section~\ref{sec:results} below.

We want to highlight one aspect of our piggyback theorems, namely the conditions  we supply which ensure that dualities are \emph{strong}. This is a  completely new feature within the   piggybacking framework, and its development is made possible by  our consideration of co-dualities, in which, loosely,  the roles of~$\M$ and~$\MT$ are swapped. This in turn hinges on the symmetry inherent in our allowing~$\M$
to be a structure rather than an algebra;  for an algebra~$\M$, a  dualising alter ego $\MT$ (with topology disregarded) will rarely be an algebra. We draw attention to our Piggyback Strong Duality Theorem~\ref{thm:pigstrong1} in which we identify   conditions under which single-carrier piggybacking yields a double best-of-all-possible-worlds scenario: we have simultaneously a duality and  a co-duality,   both strong.  It was observed long ago,  for certain well-known varieties of $\CCD$--based algebras, most notably De Morgan algebras and Stone algebras, that a natural duality (obtainable by single-carrier piggybacking) `coincides' with Priestley duality.   It is this phenomenon and its co-duality analogue that are witnessed by  the specialisation of   Theorem~\ref{thm:pigstrong1}   to $\CCD$-based algebras,
Theorem~\ref{thm:DLfullpig1}. We draw attention, too, to an immediate consequence of Theorem~\ref{thm:pigstrong1}, the Two-for-one Piggyback Strong Duality Theorem \ref{thm:two4one}.  Here the duality and co-duality are dualities in partnership, as in \cite{DHP12}, with each obtained from the other simply by topology-swapping.
Theorem~\ref{ex:Ockham} provides a nice application, in which moreover the generating algebra is infinite: we give a purely piggyback-based proof of Goldberg's 
duality for the variety of Ockham algebras, which is strong,  and show that it is in partnership with a strong co-duality.


The paper is organised in the following way.  In Section~\ref{sec:scene}
we set the scene for what follows.    Here, and subsequently, we shall
assume that the reader is familiar with the basic theory of natural dualities
as presented in \cite{NDftWA} and shall focus on  aspects of the theory
not to be found there, including some notions we introduce \textit{ab initio}.
Section~\ref{sec:results}  begins with a presentation of our Basic
Assumptions---the framework within which we establish our principal results.
These results are then set out,  prefaced by general comments on the piggybacking strategy. The proofs are postponed to Section~\ref{sec:proofs} and modularised
via a series of lemmas.  That section begins with a summary of the overall strategy and of the contribution made by  the various  lemmas. In Section~\ref{sec:proofs}, as in Section~\ref{sec:scene}, we shall where possible allow partial operations locally. 

Sections~\ref{sec:DL} and~\ref{sec:SL} do not rely explicitly on the technical material in Section~\ref{sec:proofs}. They  specialise the main  results to the two settings---$\CCD$-based algebras and $\CS$-based algebras---in which piggybacking has hitherto principally been employed. Here we have two objectives: to demonstrate how previous applications of piggybacking fit into a wider framework and to present new results, specifically on co-dualities and concerning  prevarieties with infinite generators.

We wish also to emphasise what the present paper does not cover.  Our examples
are confined to  applications  of our theorems to  piggybacking over~$\CCD$
and  over $\CS$. Piggybacking over other amenable base categories is not explored here (though we note an overture in this direction made by Cabrer and Priestley \cite{CPTWO}, using piggybacking over the variety of distributive bilattices).
There is also scope for a more comprehensive investigation of natural dualities in  the non-finitely generated case, and associated examples:  the theory here is as yet embryonic. We have addressed multi-carrier piggybacking for prevarieties
but not full-blown multisorted dualities; however we would not foresee  obstacles to  extending our theory to the multisorted case.  In conclusion, we  assert that  our achievements  in this paper  should be seen as meeting  interim  objectives rather than final ones.   Further work is expected to reveal additional insights and to contribute new examples.

\section{Setting the scene}  \label{sec:scene}

In this section we introduce the setting in which we shall work.
A basic reference for natural duality theory for structures, rather than algebras,
is Davey \cite{D06}. A  discussion of the notions introduced below can be found there.

We shall want to consider classes of structures of the form $\ISP(\M)$, where~$\M$ is a structure of a suitable kind and $\M$ is not necessarily finite.  We refer to $\ISP(\M)$ as the  \emph{prevariety} generated by~$\M$; only when $\M$ is finite is this prevariety guaranteed to be definable by quasi-equations, and so to be a quasivariety according to the normal usage of this term. We regard the prevariety $\ISP(\M)$ as a category by taking as morphisms all structure-preserving maps.
The structures we shall consider take the form $\M = \langle M; G,H,R\rangle$.  Here $G$ and~$H$ are sets of, respectively, total and partial operations on $M$, of finite arities, and~$R$ is a set of finitary relations on~$M$. We assume that the relations in $R$ and the domains of the partial operations in $H$ are non-empty. Any of $G$, $H$ and $R$ may be empty. We say that $\M$ is a \emph{total structure} if $H=\varnothing$,  that it is a \emph{total algebra}  if $H=R=\varnothing$ and
that it is \emph{purely relational} if $G = H = \varnothing$. Given a set $H$ of partial operations on a set~$M$, we define
\[
\dom(H) \coloneqq \{\, \dom(h)\mid h\in H\,\} \text{ and } \graph(H) \coloneqq  \{\, \graph(h)\mid h\in H\,\}.
\]
We note that in certain contexts (see \cite[pp.~40--41]{NDftWA} for more details) it is convenient to replace the members of $G \cup H$ by their graphs.

Let $\M =\langle M; G, H, R\rangle$ be a structure.   Define $\CA = \ISP(\M)$ to be the prevariety generated by $\M$ and let $\A\in \CA$. We say that a subset $X$ of $\CA(\A, \M)$ \emph{separates the structure} $\A$ if the following equivalent conditions hold:
\begin{enumerate}[(1)]

\item the natural map $\eta \colon A \to M^X$, given by $\eta(a)(x) \coloneqq  x(a)$, for all $a\in A$ and all $x\in X$, is an embedding of $\A$ into $\M^X$;

\item

\begin{rmlist} 

\item[(i)] 
for all $a,b\in A$ with $a\ne b$, there exists a morphism $x\in X$ with $x(a)\ne x(b)$, and

\item[(ii)]  
for every ($n$-ary) relation $r$ in $\dom(H^\A)\cup R^\A$ and all $a_1, \dots, a_n$ in $A$ with $(a_1, \dots, a_n) \notin r^\A$, there exists a morphism  $x\in X$ with $(x(a_1), \dots, x(a_n)) \notin r^{\M}$.
\end{rmlist}

\end{enumerate}
Note that $\CA(\A, \M)$ separates the structure $\A$ as $\A\in \ISP(\M)$, and
if $\M$ is a total algebra then a subset $X$ of $\CA(\A, \M)$ separates the structure $\A$ if and only if (2)(i) holds, that is, $X$ separates the points of $A$.

We shall also wish to consider structures-with-topology
of the form $\MT =  \langle M; G,H,R,\Tp\rangle$.   Here $\langle M; G,H,R\rangle $ will be assumed to be a structure of the same type as we considered above, and $\Tp$ is a compact Hausdorff topology on~$M$---a priori, no compatibility is assumed between the structure and the topology. If~$M$ is finite then $\Tp$ is necessarily discrete. We define $\CX \coloneqq  \IScP(\MT)$ to be the class of  all
structures-with-topology $\X$ of the same type as $\MT$ for which $\X$ is isomorphic to a closed substructure of a non-zero power of~$\MT$. Note that the empty structure-with-topology  $\boldsymbol\varnothing$ belongs to $\CX$ in case there are no nullary operations in the type of~$\MT$. Relations and total and partial operations are lifted pointwise from $\MT$ to any member $\X$ of $\CX$, and the domain of such a lifted map is indicated by the appropriate superscript.
We regard $\CX$ as a category by taking as morphisms all continuous structure-preserving maps. Given a structure-with-topology $\MT$, let $\M'$ be the structure obtained by removing the topology.  When properties  relating to structure are said to hold for  $\MT$ we shall mean that they  are true in $\M'$.

We now need to clarify what is meant by saying that a structure with topology
is an alter ego of a structure on the same underlying set $M$. Here we must extend to the case that~$M$ is not necessarily finite the notion of compatible structures introduced by Davey in \cite{D06}. Let $\M \coloneqq \langle M; G_1, H_1, R_1\rangle$ and $\M' \coloneqq \langle M; G_2, H_2, R_2\rangle$ be structures on the
set $M$. We say that $\M'$ \emph{is compatible with} $\M$ if, for all $n\in \mathbb N$, each $n$-ary relation in $\dom(H_2)\cup R_2$ forms a substructure of $\M^n$ and each operation in $G_2\cup H_2$ is a homomorphism with respect to~$\M$.  It is a symbol-pushing exercise to show that $\M'$ is compatible with $\M$ if and only if $\M$ is compatible with $\M'$. Given a compact Hausdorff topology $\Tp$ on $M$, we say that $\M=\langle M; G_1, H_1, R_1\rangle$ and $\MT =\langle M; G_2, H_2, R_2, \Tp\rangle$ are \emph{compatible} if
\begin{enumerate}[(a)]

\item 
$\M$ and $\M'$ are compatible, and

\item  
each $n$-ary relation in $\dom(H_1)\cup R_1$ is topologically closed  with respect to the topology $\Tp$ and each operation in $G_1\cup H_1$ is continuous with respect to the topology $\Tp$, that is, $\M_\Tp \coloneqq \langle M; G_1, H_1, R_1, \Tp\rangle$ is a \emph{topological structure}.

\end{enumerate}
A more compact, though less revealing, way to say that $\M$ and $\MT$ are compatible is to require that, for all $n\in \mathbb N$, each $n$-ary relation in $\graph(G_1\cup H_1)\cup R_1$ forms a topologically closed substructure of~$\MT^n$. If $\M$ and $\MT$ are compatible, then we say that $\MT$ is an \emph{alter ego} of~$\M$.

At the finite level, because the topology on $M$ is discrete, we can swap the topology to the other side, that is, $\langle M; G_1, H_1, R_1, \Tp\rangle$ is an alter ego of $\langle M; G_2, H_2, R_2\rangle$ provided $\langle M; G_2, H_2, R_2, \Tp\rangle$ is an alter ego of $\langle M; G_1, H_1, R_1\rangle$. This symmetry fails in general if $M$ is infinite.

Given a structure $\M$ and an alter ego $\MT$ for~$\M$, we can set up a dual adjunction $\langle \mathrm D, \mathrm E, e, \varepsilon\rangle$ between $\CA \coloneqq  \ISP(\M)$  and~$\CX= \IScP(\MT)$. The claims below are taken from \cite{D06}. However, their verification is straightforward (in the case that $\M$ is a finite total  algebra we refer to~\cite[Section~1.5]{NDftWA} or to~\cite[Section~1.3]{PD05} for the details). We define contravariant hom-functors $\mathrm D \colon\CA \to \CX$ and $\mathrm E \colon  \CX \to \CA$ as follows:
\begin{alignat*}{3}
\text{on objects}& \qquad\qquad & \D\A &= \CA(\A,\M), & \qquad\qquad \text{\phantom{on objects}}& \\
\text{on morphisms}& & \D f &= - \circ f; & \text{\phantom{on morphisms}}& \\
\shortintertext{and}
\text{on objects}& & \E \X &= \CX(\X,\MT), & \text{\phantom{on objects}}& \\
\text{on morphisms}& & \E\psi &= - \circ \psi. & \text{\phantom{on morphisms}}
\end{alignat*}
The well-definedness of these functors is a consequence of our compatibility assumption. For each structure $\A\in \CA$, the hom-set $\CA(\A,\M)$  forms a closed substructure of $\MT^A$ and so $\mathrm {D}(\A)$  (the \emph{dual of} $\A$) is a member of~$\CX$. Likewise, for each structure $\X\in \CX$, the set $\CX(\X,\MT)$   forms a substructure $\E\X$ of $\M^X$  and so $\E \X$ (the \emph{dual of}~$\X$) is a member of~$\CA$.

For each $\A\in \CA$ and each $\X\in \CX$, we define the \emph{evaluation maps}
 \[
\esubA  \colon\A \to \ED \A \quad \text{and} \quad \epsub \X \colon \X \to
\DE \X
 \]
by $\eA a(x) \coloneqq  x(a)$, for all $a\in A$ and all $x\in \CA(\A,\M)$,
and $\epsub \X (x)(\alpha) \coloneqq  \alpha(x)$, for all $x\in X$ and all
$\alpha\in \CX(\X,\MT)$. Then $\langle \mathrm D, \mathrm E, e, \varepsilon\rangle$  is a dual adjunction between $\CA$ and~$\CX$ and $e \colon \id_{\CA} \to \mathrm{ED}$ and $\varepsilon \colon \id_{\CX} \to \mathrm{DE}$ are natural transformations. Moreover, the construction of $\CA$ and $\CX$ via $\ISP$ and $\IScP$, respectively, ensures that the maps $\esubA  \colon  \A \to \ED \A$ and $\epsub \X \colon \X \to \DE \X$ are embeddings, for all $\A\in \CA$ and all $\X\in \CX$. We note that here an \emph{embedding} in $\CA$ means `isomorphism onto a substructure' while in $\CX$ it means `isomorphism onto a topologically closed substructure'.

The following definitions mimic those given in \cite{NDftWA} for the special case in which $\M$ is a finite total algebra. The alter ego $\MT$ \emph{yields a duality on~$\CA$},  or more briefly $\MT$ \emph{dualises $\M$}, if the map $\esub \A$ is
an isomorphism, for all $\A\in \CA$. If the natural map $\epsub\X \colon \X \to \DE \X$ is an isomorphism for all $\X\in\CX$, then we say that $\MT$ \emph{yields a co-duality on $\CA$}, or $\MT$ \emph{co-dualises} $\M$; if the emphasis is on the structure with topology, we say that $\M$ \emph{yields a duality on $\CX$}, or $\M$ \emph{dualises}~$\MT$. The alter ego $\MT$ \emph{yields a full duality on~$\CA$}, or more briefly $\MT$ \emph{fully dualises $\M$}, if it yields both a duality and a co-duality on~$\CA$. In this case the functors $\mathrm D$ and $\mathrm E$ give a dual equivalence between the categories $\CA$ and~$\CX$.

While the notions of dualising and fully dualising alter ego
parallel exactly those given in \cite{NDftWA},  more care is needed when  extending from algebras to structures the concept of a strong duality as presented
in~\cite[Chapter~3]{NDftWA}. We recall that, in the restricted setting,  proving that a duality is strong has been the primary tool for establishing that the duality is full. For structures in general, there are two competing definitions for a full duality yielded by an alter ego~$\MT$ to  be a strong duality:
\begin{enumerate}

\item[(1)]  
every closed substructure of a non-zero power of $\MT$ is hom-closed, or equivalently is term-closed, and

\item[(2)]
$\MT$ is injective in the category~$\CX$ (with respect to embeddings).
\end{enumerate}
Fortunately, the two definitions coincide when~$\M$ is a total structure, as it always will be in our theorems. Accordingly,  we shall say that $\MT$ \emph{yields a strong duality on~$\CA$}, or that $\MT$ \emph{strongly dualises $\M$}, if~(2) holds.
(If   partial operations are permitted in the type of~$\M$, then (2) is too strong, and (1) is the more appropriate definition. See the discussion in \cite[Section~4.2]{DPW11}.)

 So far, the definitions we have given are those applicable to duality theory for structures.  They are not specific to dualities of piggyback type.  Our principal theorems involve extensions to a much wider setting of conditions which underpin the piggybacking method for prevarieties of algebras.  We now introduce the requisite definitions, allowing also for the topological  versions we shall need in order to obtain both duality and co-duality theorems.

Given a structure $\M$ we denote the set of \emph{total} unary term functions of $\M$ by $\Clo_1(\M)$. (Note that a total unary term function of $\M$ may result from composing  operations and partial operations in the type of~$\M$.) If $\MT$ is an alter ego of $\M$, then the compatibility between $\M$ and $\MT$ guarantees that $\Clo_1(\MT) \subseteq \End(\M)$ and $\Clo_1(\M) \subseteq \End(\MT)$. The compatibility also guarantees that the dual of the free structure in $\CA$ on one generator is isomorphic to $\MT$; it follows that if $\MT$ dualises $\M$, then 
$\Clo_1(\M) = \End(\MT)$---see the discussion in~\cite[p.~13]{D06}. If $\MT$ fully dualises $\M$, then the topologically closed substructure of $\MT^M$ generated by $\Clo_1(\MT)$ is $\End(\M)$ (so $\Clo_1(\MT) = \End(\M)$ if $M$ is finite)---the proof is a simple generalisation of that given for the finite case in~\cite[Proposition~4.3]{DHW05};  see also \cite[Corollary~3.9]{DHW05b} and \cite[Lemma~4.3]{DPW11}.

We next extend to structures the idea of a term-reduct of a total algebra.
For further details see~\cite{DPW11}. Let $\M = \langle M; G, H, R\rangle$ be a structure. For $n\in \mathbb N$, an $n$-ary relation $r$ on $M$ is \emph{conjunct-atomic definable from $\M$} if
\[
r = \{\, (a_1, \dots, a_n) \in M^n \mid \M \models \textstyle\bigwith\limits_{i=1}^m \alpha_i(a_1,\dotsc, a_n)\,\},
\]
where $v_1,\dotsc,v_n$ are distinct variables and each $\alpha_i(v_1,\dotsc, v_n)$ is an atomic formula in the language $\langle G, H, R\rangle$ involving variables from the set $\{v_1,\dotsc,v_n\}$. A structure $\M' = \langle M; G', H', R'\rangle$ is a \emph{structural reduct} of $\M$ if  each relation in $\dom (H')\cup R'$ is conjunct-atomic definable from $\M$, each $g\in G'$ belongs to the enriched partial clone of~$\M$ and each $h\in H'$ has an extension in the enriched partial clone of~$\M$; see~\cite[Lemma~2.5]{DPW11}. Now assume that $\M$ has a structural reduct $\M^\flat$ in a prevariety~$\CB \coloneqq \ISP(\N)$, where $\N =\langle N; G^\nu, H^\nu, R^\nu\rangle$ is some structure.  Then every structure $\A$ in $\CA = \ISP(\M)$ also has a structural reduct $\A^\flat = \langle A; G^\nu, H^\nu, R^\nu\rangle$ in $\CB$.
We can define the structure $\A^\flat$ in two equivalent ways: either syntactically via the conjunct-atomic formulas and terms that define $\M^\flat$ from $\M$ or semantically via the embedding of $\A$ into a power of $\M$. Either way, we have $\A^\flat \in \ISP(\M^\flat)$. When we say that \emph{$\MT = \langle M; G, H, R, \Tp\rangle$ has a structural reduct $\MT^\flat = \langle M; G^\nu, H^\nu, R^\nu, \Tp\rangle$ in $\IScP(\NT)$}, we mean that $\M' \coloneqq  \langle M; G^\nu, H^\nu, R^\nu\rangle$ is a structural reduct of $\langle M; G, H, R\rangle$ and $\MT^\flat$ belongs to $\IScP(\NT)$---it is not sufficient to know that $\M'$ belongs to $\ISP(\N')$, where $\N'$ denotes $\NT$ minus its topology. If $\MT$ has a structural reduct in $\IScP(\NT)$, then every member $\X$ of $\IScP(\MT)$ has a structural reduct $\X^\flat$ in~$\IScP(\NT)$. In fact, in the non-topological case, $^\flat$ is a functor from $\ISP(\M)$ to $\ISP(\N)$, and in the topological case, $^\flat$ is a functor from $\IScP(\MT)$ to $\IScP(\NT)$. 

Assume that  the structure  $\NT$ is an alter ego of the structure $\N$ and that $\NT$  is injective in~$\CY\coloneqq  \IScP(\NT)$. (In fact, injectivity at the finite level would suffice.)  The fact that $\N$ and $\NT$ are compatible guarantees that $\N$ is structurally equivalent to a total structure, and may without loss of generality be taken to be a total structure. This observation  justifies  an assumption we make in the next section.

The notion of entailment will be important in our theory, as it is in the setting of quasivarieties of  algebras \cite[Section~2.4 and Chapter~8]{NDftWA}. Assume that $\MT$ is an alter ego of a structure $\M$ and define $\CA \coloneqq \ISP(\M)$ and $\CX \coloneqq \IScP(\MT)$. Let $r\subseteq M^n$, for some $n\in \mathbb N$. We say that \emph{$\MT$ entails $r$} if $r$ forms a substructure of $\M^n$ and, for all $\A\in \CA$, every $\CX$-morphism $\alpha \colon \CA(\A, \M) \to \MT$ preserves $r$. Likewise, we say that \emph{$\M$ entails $r$} if $r$ forms a topologically closed substructure of $\MT^n$ and, for all $\X\in \CX$, every $\CA$-morphism $u \colon \CX(\X, \MT) \to \M$ preserves $r$.  In connection with entailment we shall encounter the diagonal $\Delta_M \coloneqq  \{\, (a,a) \mid a \in M\,\}$, which forms a substructure of $\M^2$.
 
Those familiar with traditional piggybacking will find the next  definition
unsurprising. Let $\M$ be a structure, let $N$ be a set, and let $r\subseteq N^n$, for some $n\in \mathbb N$. For $\omega_1, \dots, \omega_n\colon M \to N$,  define
\[
(\omega_1,\dots, \omega_n)^{-1}(r) \coloneqq 
\{\, (a_1,\dots, a_n) \in M^n  \mid (\omega_1 (a_1), \dots, \omega_n (a_n))\in r\,\}.
\]
Given $\Omega\subseteq N^M$, define
\begin{multline*}
\Omegamax \M r : = \{\, s \subseteq M^n \mid \mathbf s \le \M^n \text{ with $s$ maximal in }\\ (\omega_1,\dots, \omega_n)^{-1}(r) \text{ for some } \omega_1,\dots, \omega_n\in \Omega\,\}
\end{multline*}
(here, and subsequently,  $\le$ denotes `is a substructure of').

\begin{remark} \label{rem:rnotinDeltaM}
As above, let $\omega $ be a map from~$M$ to~$N$. We shall encounter in the statements of certain of our theorems various entailment conditions relating to relations contained in $\ker(\omega) = (\omega,\omega)^{-1} (\Delta_N)$. Since $\Delta_M \subseteq \ker(\omega)$ and $\Delta_M$ forms a substructure of $\M^2$,  we deduce that if $r\subseteq\Delta_M$ and $r$ forms a substructure of $\M^2$ that is maximal in $\ker(\omega)$, then $r= \Delta_M$ and $\M$ entails~$r$ trivially.
Similar considerations apply when $\M$ is replaced by~$\MT$.  It follows that,
in all occurrences of  entailment conditions involving binary relations maximal in $\ker(\omega)$, we may without loss of generality restrict to relations not contained in $\Delta_M$.
\end{remark}

For a structure $\MT$ with a topology, we define
\[
\Omegamax {\stwiddle M} r : = \Omegamax {\M'} r,
\]
where $\M'$ is $\MT$ with its topology removed.
Lemma~\ref{lem:MaxImpliesClosed}  guarantees that under minimal assumptions the
elements of sets $\Omegamax {\stwiddle M} r$ and of $\Omegamax {\stwiddle M} {\Delta_N}$ are  \emph{closed} substructures of $\M^n$, so that it will make sense to demand  that such relations are entailed by~$\M$.   This fact is tacitly used in various theorem statements later on, and in  those lemmas in Section~\ref{sec:proofs}  in which entailment by~$\M$ arises.

\begin{lemma}[Closed Maximal Relations Lemma]
\label{lem:MaxImpliesClosed}
Let $\M =\langle M; G, R\rangle$ be a total structure and assume that $\Tp$ is a topology on $M$ such that each $g\in G$ is continuous with respect to $\Tp$. Let $\N = \langle N; \Tp\rangle$ be a topological space, and let $r$ be a closed subset of $\N^n$, for some $n\in \mathbb N$. If $\omega_1, \dots, \omega_n\colon M \to N$ are continuous  with respect to the topologies on $M$ and $N$, then every relation $s\subseteq M^n$ that forms a substructure of $\M^n$ and is maximal in $(\omega_1,\dots, \omega_n)^{-1}(r) $ is topologically closed.
\end{lemma}

\begin{proof}
Assume that $\omega_1, \dots, \omega_n\colon M \to N$ are continuous and let $s$ be a subset of $M^n$ that forms a substructure of $\M^n$ and is maximal in $(\omega_1,\dots, \omega_n)^{-1}(r)$. Since $r$ is closed, it follows that $(\omega_1,\dots, \omega_n)^{-1}(r)$ is also closed. Hence the closure $\overline s$ of $s$ is a subset of $(\omega_1,\dots, \omega_n)^{-1}(r)$.
Since~$s$ is a substructure of~$\M^n$, it is closed under each~$g$ in~$G$. By continuity of~$g$, the set $\overline s$ is also closed under~$g$. 
Hence $\overline s$ forms a substructure of $\M^n$. The maximality of $s$ implies that $\overline s = s$, whence $s$ is topologically closed.
\end{proof}

The remainder of this section concerns notions which will not arise subsequently until we reach our co-duality theorems, and the results dependent on these.
Because we are using the usual setting, in which the empty structure is allowed in $\CX = \IScP(\MT)$ (when the type of $\MT$ includes no nullary operations) but is not permitted in $\CA =\ISP(\M)$, we need to take
 care with naming constants in the type of~$\MT$ when considering co-dualities.   Since it simplifies the discussion and is all we need here, we shall restrict ourselves to the situation where both $\M$ and $\MT$ are total structures.
We say that a structure~$\A$ has \emph{named constants} if the value of every constant unary term function of~$\A$ is the value of a nullary term function. (Of course, to guarantee that~$\A$ has named constants, it
suffices to make the value of just one constant term function into the value of a nullary operation.) We require another definition. Given a total structure $\M = \langle M; G, R\rangle$ and $a\in M$, we say that $\{a\}$ forms a \emph{complete one-element substructure of $\M$} if $g(a, \dots, a) = a$, for all $g\in G$ and $(a, \dots, a)\in r$, for all $r\in R$.
Define $\C_1$ to be the topologically closed substructure of $\MT$ generated by the set of values of the constant unary term functions.
If $\MT$ has constant unary term functions but no nullary operations, then the empty structure $\boldsymbol\varnothing$ belongs to $\CX$ and it is easy to see that $\E {\boldsymbol\varnothing}$ and $\E{\C_1}$ are isomorphic complete one-element structures, so it is impossible for $\MT$ to co-dualise~$\M$. (It also follows that $\DE {\C_1}$ is in a one-to-one correspondence with the complete one-element
substructures of~$\M$.) Consequently, $\MT$ having named constants is a necessary condition for $\MT$ to co-dualise~$\M$. Indeed, this assumption guarantees that $\epsub {\boldsymbol\varnothing} \colon \boldsymbol\varnothing \to \DE {\boldsymbol\varnothing}$ is an isomorphism whenever $\boldsymbol\varnothing$ belongs to~$\CX$. The other conditions we give in our co-duality theorems are sufficient to prove that $\epsub \X \colon \X \to \DE \X$ is an isomorphism, for all non-empty $\X$ in~$\CX$. For a detailed discussion of named constants in the finite case, see~\cite[Lemma~6.1]{D06} and~\cite[Lemma~3.1.2]{NDftWA}.

In some of our results, we shall not need
to postulate explicitly that $\MT$ has named constants:  the following lemma ensures that this will follow
from other assumptions we shall
make.

\begin{lemma}[Named Constants Lemma]\label{lem:liftconstants}
Let $\M$ and $\N$ be non-trivial total structures \textup(not necessarily of the same type\textup), and let $\MT$ and $\NT$ be total structures that are alter egos of $\M$ and~$\N$, respectively.
Define $\CB\coloneqq \ISP(\N)$ and $\CY \coloneqq  \IScP(\NT)$, denote the induced dual adjunction between $\CB$ and $\CY$ by $\langle \mathrm{H}, \mathrm{K}, k, \kappa\rangle$, and assume that $\NT$ co-dualises~$\N$.
Assume that $\M$ and $\MT$ have structural reducts $\M^\flat$ in $\CB$ and  $\MT^\flat$ in~$\CY$, respectively. Then $\MT$ has named constants.
\end{lemma}

\begin{proof}
Let $a\in M$ and assume that $\MT$ has a constant unary term function $t_a$ with value~$a$. Since $\MT$ is compatible with $\M$, the map $t_a$ is an endomorphism of $\M$ and hence $\{a\}$ forms a complete one-element substructure of $\M$---here we use the fact that each relation $r$ in the type of $\M$ is non-empty.
Hence $\{a\}$ forms a complete one-element substructure of $\M^\flat$.
As $\M^\flat$ is nontrivial and $\M^\flat$ belongs to $\CB=\ISP(\N)$, there exists $\omega\in \CB(\M^\flat, \N)$.
As $\{a\}$ forms a complete one-element substructure of $\M^\flat$, it follows that $\omega(a)$ forms a complete one-element substructure of $\N$. Let $\C_1$ denote the topologically closed substructure of $\NT$ generated by the set of values of the constant unary term functions of~$\NT$. As we noted above,
$\mathrm{HK}(\C_1)$
is in a one-to-one correspondence with the complete one-element substructures of~$\N$ and so
$\mathrm{HK}(\C_1)$
is non-empty. Since $\NT$ co-dualises~$\N$, the map $\epsub {\C_1} \colon \C_1 \to
\mathrm{HK}(\C_1)$
is an isomorphism and consequently $\C_1$ is non-empty. But $\NT$ has named constants, again since $\NT$ co-dualises~$\N$, and so the type of $\NT$ includes a nullary operation $\sigma$. Since $\MT^\flat$ is a structural reduct of $\MT$ in~$\CY$, it follows that $\MT$ has a nullary term function $s_a$ with value $\sigma^{\stwiddle M^\flat}$. Hence, $\MT$ has named constants.
(In particular,
$t_a(s_a)$ is a nullary term function of $\MT$ with value~$a$.)
\end{proof}

\section{The Piggyback Duality Theorems} \label{sec:results}

Now we set up the  framework of Basic Assumptions, within which we shall develop our piggyback duality and co-duality theorems, and the associated notation.
The benefits in terms of simplicity we gain by working uniformly with the Basic Assumptions outweigh the extra generality  gained at the margins by tailoring each theorem to a minimal set of assumptions.
For example, injectivity of~$\NT$ in~$\CY$ is explicitly needed in the
Piggyback Duality Theorem but not in the Piggyback Co-duality Theorem,
whereas injectivity of~$\N$ in~$\CB$ is needed in the latter theorem but not the former.
Thus the conditions laid out in the Basic Assumptions  are a little stronger than we strictly need in every case.
For this reason, in Section~\ref{sec:proofs}, where we give the proofs of our theorems, we shall include in each lemma only the conditions needed in the proof.

\begin{assume}\label{BasicAss}\
\begin{itemize}

\item
$\M$ and $\N$ are non-trivial total structures (not necessarily of the same type), and  $\MT$ and $\NT$ are total structures that are alter egos of $\M$ and $\N$, respectively.

\item
$\CA\coloneqq \ISP(\M)$ and $\CX\coloneqq \IScP(\MT)$ and the induced dual adjunction be\-tween $\CA$ and~$\CX$
is denoted by $\langle \mathrm{D}, \mathrm{E}, e, \varepsilon\rangle$.

\item
$\CB\coloneqq \ISP(\N)$ and $\CY \coloneqq  \IScP(\NT)$  and the induced dual adjunction  between $\CB$ and $\CY$ is denoted $\langle \mathrm{H}, \mathrm{K}, k, \kappa\rangle$. In addition,
\begin{enumerate}[\normalfont (i)]
\item  $\NT$ yields a full duality on $\CB$;
\item  $\NT$ is injective in~$\CY$ and $\N$ is injective in $\CB$.
\end{enumerate}
\end{itemize}
We shall add additional conditions as needed.  In  particular we shall, theorem by theorem, impose appropriate relationships between $\M$ and~$ \N$ or between $\MT$ and $\NT$.

We adopt the following notation.
Let $\Omega\subseteq N^M$ and let $\mathcal E\subseteq M^M$.
We write
\[
\Omega\circ \mathcal E \coloneqq  \{\, \omega \circ u \mid \omega\in \Omega \And u \in \mathcal E\,\}
\subseteq N^M.
\]
If $\Omega = \{\omega\}$, then we write simply $\omega\circ \mathcal E$.
\end{assume}

Operating under our Basic Assumptions, we now  briefly review  the ideas behind
 the piggybacking method, as presented in \cite{DW85,DW83PB,DP87}.
We assume that the structure $\M$, for which we wish to find a dualising alter ego, has a structural reduct $\M^\flat$ in the category $\CB$, for which we already have a full duality.
An alter ego $\MT$ will dualise $\M$ provided that, for each $\A \in \CA$, we can find a one-to-one map $\alpha \mapsto d_\alpha$ from $\mathrm{ED}(\A)$ to $\mathrm{KH}(\A^\flat)$ that commutes with the evaluations, that is,
$d_{-}\circ \esubA  = k_{\A^\flat}$; see Figure~\ref{fig:pig-fig1}.
\begin{figure}[h!t]
\begin{center}
\begin{tikzpicture}
[auto,
 text depth=0.25ex]
\matrix[row sep= 1.6cm, column sep= 1.6cm]
{
 \node (A) {$\A$};
  &\node (EDA)  {$\mathrm{ED}(A)$};
\\  & \node (KHA) {$\mathrm{KH}(\A^\flat)$}; \\
};
\draw [-latex] (A) to node {$\esubA $}  (EDA);
\draw [-latex] (A) to node [swap,xshift=.2cm] {$k_{\A^\flat}$}  (KHA);
\draw [-latex] (EDA) to node {$d_{-}$}  (KHA);
\end{tikzpicture}
\caption{The commuting triangle: the equation
$d_{-}\circ \esubA  = k_{\A^\flat}$}\label{fig:pig-fig1}
\end{center}
\end{figure}

We now come to the key idea underlying piggybacking: how we might construct the
map~$d_{-}$ such that $d_{-}\circ \esubA  = k_{\A^\flat}$, as in the commuting triangle diagram in Figure~\ref{fig:pig-fig1}.
We seek to use an element $\omega$, or a set $\Omega$ of elements,
in  $\CB(\M^\flat, \N)$ to link $\ED \A$ to $\mathrm{KH}(\A)$, for each
$\A \in \A$.  Following the usage in \cite{DP87} we shall  refer to such  maps as \emph{carriers}.  Explicitly, given $\omega\in \CB(\M^\flat, \N)$ and $\A\in \CA$, we may define a map
\[
\Phi_\omega^\A \colon \CA(\A, \M) \to \CB(\A^\flat, \N)
\]
by $\Phi_\omega^\A(x) \coloneqq \omega \circ x$, for all $x \in \CA(\A, \M)$.
In a best-case scenario there will exist a single carrier~$\omega$ making
$\Phi_\omega^\A$ surjective for each $\A \in \CA$.
In that case, given $\alpha\colon \CA(\A,\M)\to M$, we have a chance of  ``defining"
$d_\alpha \colon \CB(\A^\flat, \N) \to N$  by $d_\alpha (\omega \circ x) = \omega(\alpha(x))$, for $x \in \CA(\A,\M)$.
The  success of this construction then depends on our being able to select
$\omega$ and the alter ego $\MT$ in such a way that~$d_{-}$ is well defined and
one-to-one.

Historically, the piggybacking technique was devised in two stages. What we  call single-carrier piggybacking was formulated by Davey and Werner \cite{DW85,DW83PB}.  In \cite{DP87}, Davey and Priestley initiated the theory of multisorted natural dualities and thereby showed how to dispense with the restriction that there be
a single carrier, in the following way.
Suppose that we could find some subset~$\Omega $ of $\CB(\M^\flat,\N)$
such that, for each~$\A$,  the family $\{\Phi_\omega^\A\}_{\omega\in\Omega}$
is \emph{jointly surjective} in the sense that the union of the images of the maps $\Phi_\omega^\A$, for $\omega \in \Omega$, is the whole of $\CB(\A^\flat,\N)$.  Then we may still try to ``define'' $d_\alpha$ as before, but now with~$\omega$, as well as~$x$, varying. Presuming joint surjectivity, we would  then need to choose $\Omega $ and the alter ego so that~$d_{-}$ is well defined (and this is more of an issue than when there is a single~$\omega$) and one-to-one.

We are finally ready to state our piggyback duality and co-duality theorems.
We state each theorem first for the single-carrier case.  We single this out
for two reasons.  Firstly, the conditions we give for single-carrier piggybacking to work are not merely specialisations of those for the general case. Secondly,  our
strong-duality results rely only  on the  single-carrier versions of the duality and co-duality theorems.  So too do our results on semilattices  in Section~\ref{sec:SL}.

We note that in Condition~(3)(ii)(a) of Theorems~\ref{thm:pig-simple}
and Theorem~\ref{thm:copig-simple} we could insert the restriction that
$r \nsubseteq \Delta_M$; recall Remark~\ref{rem:rnotinDeltaM}.  Similar restrictions can be imposed in other theorems, likewise.

\begin{theorem}[Single-carrier Piggyback Duality Theorem] \label{thm:pig-simple}
The setting for this theorem is provided by the Basic Assumptions.
Assume that $\M$ has a structural reduct $\M^\flat$ in $\CB$ and let $\omega \in \CB(\M^\flat, \N)$. Assume that $\NT = \langle N;G^\nu,R^\nu,\Tp\rangle$.
Then $\MT$ dualises $\M$ provided
{\upshape (0)--(3)} below all hold.

\begin{enumerate}[\normalfont (1)]

\item[\normalfont(0)]
$\omega$
is continuous with respect to the topologies on $\MT$ and $\NT$.

\item
$\omega\circ \Clo_1(\MT)$ separates the structure~$\M^\flat$.

\item One of the following conditions holds:

\begin{rmlist}
 \item[\normalfont(i)]  $\NT$ is purely relational, or

\item[\normalfont(ii)]
$\MT$ has a
structural
reduct $\MT^\flat$ in~$\CY$
with
$\omega\in
\CY(\MT^\flat, \NT)$.

\end{rmlist}

\item
\begin{rmlist}
 \item[\normalfont (i)] The structure $\MT$ entails every relation in $\omegamax {\M} r$,
for each relation
$r\in R^\nu$,
and

\item[\normalfont (ii)] 
either
\begin{rmlist}
\item[\normalfont (a)] $\MT$ entails each binary relation $r$ on $M$
which forms a substructure of $\M^2$ that is maximal in $\ker(\omega)$, or

\item[\normalfont (b)]
$\omega \circ \Clo_1(\M)$ separates the points of $M$.

\end{rmlist}
\end{rmlist}
\end{enumerate}
\end{theorem}

In the above theorem, Condition (3)(ii)(b) implies Condition (3)(ii)(a), but it is convenient to have (3)(ii)(b) recorded explicitly.
(Indeed, it is an almost trivial exercise to show that if $\omega \circ \Clo_1(\MT)$ separates the points of $M$, then each binary relation on $M$
 that forms a substructure of $\M^2$ and is contained in $\ker(\omega)$ is contained in $\Delta_M$. A maximal such relation must therefore be $\Delta_M$, which is trivially entailed by any alter ego.)


We now give the Piggyback Duality  Theorem in its general form.
It is worth drawing attention to the structure of the statement of the theorem.
As occurs in  Theorem~\ref{thm:pig-simple} too,  we have a
list of conditions (0)--(3), which together with the initial assumptions yield the conclusion.  In  Condition (2) there is a dichotomy:
(2)(i) covers the case in which $\NT$ is purely relational, whereas (2)(ii) allows
for operations  (albeit at most unary) in $\NT$ but then imposes a  rather stringent
condition  on how $\MT$ is related to $\NT$.  An analogous dichotomy appears in
Condition~(2) in the Piggyback Co-duality Theorem.

\begin{theorem}[Piggyback Duality Theorem] \label{thm:pig-general}
The setting for this theorem is provided by the Basic Assumptions.
Assume that $\M$ has a structural reduct $\M^\flat$ in $\CB$ and let  $\Omega$ be a subset of $ \CB(\M^\flat, \N)$.
Assume that $\NT = \langle N;G^\nu,R^\nu,\Tp\rangle$.
Then $\MT$ dualises $\M$ provided
either $\Omega$ is a singleton set $\{\omega\}$ and the conditions for
Theorem~{\upshape\ref{thm:pig-simple}} are met or
{\upshape (0)--(3)} below all hold.

\begin{enumerate}[\normalfont (1)]

\item[\normalfont(0)]
$\Omega$ is finite and each $\omega\in \Omega$ is continuous with respect to the topologies on $\MT$ and $\NT$.

\item
$\Omega\circ \Clo_1(\MT)$ separates the structure~$\M^\flat$.

\item One of the following conditions holds:

\begin{rmlist}
\item[\normalfont(i)]  
$\NT$ is purely relational, or

\item[\normalfont(ii)]
$\MT$ has a structural reduct $\MT^\flat$ in~$\CY$, each $\omega\in \Omega$ belongs to $\CY(\MT^\flat, \NT)$, and every operation in $G^\nu$ is unary or nullary.
\end{rmlist}

\item

\begin{rmlist}
 \item[\normalfont (i)] 
The structure $\MT$ entails every relation in $\Omegamax {\M} r$, for each relation
$r\in R^\nu$, and

 \item[\normalfont (ii)]  
$\MT$ entails every relation in $\Omegamax {\M}{\Delta_N}$.
\end{rmlist}

\end{enumerate}
\end{theorem}

We now move on to consider co-dualities.  We seek, {\it mutatis mutandis},
to mimic  the  strategy we described  in our brief survey of traditional piggybacking.
If $\MT$ has a structural reduct $\MT^\flat$ in $\CY$, then
each
$\X\in \CX$ has a structural reduct $\X^\flat$ in $\CY$,
and if $\omega\in \CY(\MT^\flat, \NT)$ and $\X\in \CX$, we may define a map
\[
\Psi_\omega^\X \colon \CX(\X, \MT) \to \CY(\X^\flat, \NT)
\]
by $\Psi_\omega^\X(\alpha) \coloneqq \omega \circ \alpha$, for all $\alpha \in \CX(\X, \MT)$.
The usual calculations show that
$\Psi_\omega^{-}$ is a natural transformation between the set-valued hom-functors $\CX(-,\MT)$ and $\CY((-)^\flat, \NT)$ out of $\CX$,
and similarly $\Phi_\omega^{-}$, as defined earlier, is a natural transformation between $\CA(-,\M)$ and $\CB((-)^\flat, \N)$.

\begin{theorem}[Single-carrier Piggyback Co-duality Theorem] \label{thm:copig-simple}
The setting for this theorem is provided by the Basic Assumptions with the additional assumption that $\MT$ has named constants.
 Assume that $\MT$ has a structural reduct $\MT^\flat$ in $\CY $ and let $\omega \in \CY(\MT^\flat, \NT)$. Assume that $\N =\langle N; G^\nu, R^\nu\rangle$.
Then $\MT$ co-dualises $\M$ provided {\upshape (1)--(3)} below all hold.

\begin{enumerate}[\ \normalfont (1)]

\item
$\omega\circ \Clo_1(\M)$ separates the structure~$\MT^\flat$.

\item
One of the following conditions holds:

\begin{rmlist}

\item[\normalfont(i)] $\N$ is purely relational, or

\item[\normalfont(ii)]  $\M$ has a structural reduct $\M^\flat$ in~$\CB$ such that  $\omega
\in \CB(\M^\flat, \N)$.

\end{rmlist}

\item The structures $\M = \langle M; G_1, R_1
\rangle$ and $\MT = \langle M; G_2, R_2,\Tp\rangle$
satisfy the following conditions:

\begin{rmlist}

\item[\normalfont(i)]

\begin{alist}

\item[\normalfont(a)]  
the operations in $G_2$ are continuous, and

\item[\normalfont(b)]
the structure $\M$ entails every relation in $\omegamax {\stwiddle M} r$, for each
$r\in R^\nu$.

\end{alist}

\item[\normalfont(ii)]
Either

\begin{alist}

\item[\normalfont(a)]
$\M$ entails each binary relation $r$ on $M$ which forms a substructure of $\MT^2$ that is maximal in $\ker(\omega)$, or

\item[\normalfont(b)]
$\omega \circ \Clo_1(\MT)$ separates the points of $M$.

\end{alist}

\end{rmlist}
\end{enumerate}

\end{theorem}

Regarding  Condition (3)(i)(a) we note that continuity of the operations in~$G_2$
is not a consequence of $\MT$'s  being an alter ego of~$\M$, as one might perhaps be tempted to assume.

\begin{theorem}[Piggyback Co-duality Theorem] \label{thm:copig-general}
The setting for this theorem is provided by the Basic Assumptions
with the additional assumption that $\MT$ has named constants.
Assume that $\MT$ has a structural reduct $\MT^\flat$ in $\CY$ and let $\Omega$ be a subset of $ \CY(\MT^\flat, \NT)$.
Assume that $\N =\langle N; G^\nu, R^\nu\rangle$.
Then $\MT$ co-dualises $\M$ provided
either $\Omega$ is a singleton set $\{\omega\}$ and the conditions of
Theorem~{\upshape\ref{thm:copig-simple}} are met or {\upshape (1)--(3)} below all hold.

\begin{enumerate}[\normalfont (1)]

\item
$\Omega\circ \Clo_1(\M)$ separates the structure~$\MT^\flat$.

\item
One of the following conditions holds:

\begin{rmlist}

\item[\normalfont(i)]
$\N$ is purely relational, or

\item[\normalfont(ii)]  $\M$ has a structural reduct $\M^\flat$ in~$\CB$, each $\omega\in \Omega$ belongs to $\CB(\M^\flat, \N)$, and every operation in $G^\nu$ is unary or nullary.

\end{rmlist}

\item The structures $\M = \langle M; G_1, R_1
\rangle$
and $\MT = \langle M; G_2, R_2,\Tp\rangle$
satisfy the following conditions:

\begin{rmlist}

\item[\normalfont(i)]
\begin{alist}

\item[\normalfont(a)] the operations in $G_2$ are continuous, and
\item[\normalfont(b)]
the structure $\M$ entails every relation in $\Omegamax {\stwiddle M} r$, for each
$r\in R^\nu$; 

\end{alist}

\item[\normalfont(ii)]
 $\M$ entails every relation in $\Omegamax {\stwiddle M}{\Delta_N}$.

\end{rmlist}
\end{enumerate}

\end{theorem}

The Piggyback Co-duality Theorem can be applied in two quite different ways. It can be used to prove that a duality given by the Piggyback Duality Theorem is in fact full. It is used in this way in the $\CCD$-based Piggyback Strong Duality Theorem~\ref{thm:DLfullpig1} and therefore also in its specialisation to Ockham algebras,
Theorem~\ref{ex:Ockham}. Alternatively, it can be used in a topology-swapping situation, where we take a duality that has been obtained via the Piggyback Duality Theorem and swap the topology from the `relational' category to the `algebraic' category.
 This leads to the $\CCD$-based and $\CS$-based Piggyback Co-duality Theorems~\ref{thm:DLcopig-general} and~\ref{thm:SLcopig}.

In certain applications, in particular to $\CCD$-based prevarieties, it may be convenient to make use of the following consequences of the piggyback duality and co-duality theorems, stated as a single corollary.  In it, we do as we shall do
subsequently and flag the duality and co-duality  versions with~(D) and~(coD).
The corollary follows from Lemma~\ref{lem:sublemma}. 

\begin{corollary}  \label{cor:sqsubseteq}
The setting is that of  the Basic Assumptions.
\begin{Dlist}
\item[\normalfont (D)]
Assume that $\M$ has a structural reduct $\M^\flat$ in $\CB$ and let  $\Omega$ be a subset of $ \CB(\M^\flat, \N)$.
Assume that $\NT = \langle N;G^\nu,R^\nu,\Tp\rangle$. Then $\MT$ dualises $\M$ if
Conditions {\upshape (0), (1), (2)} and {\upshape (3)(i)} of Theorem {\upshape \ref{thm:pig-simple}} \textup(with $\Omega= \{\omega\}$\textup) or of
Theorem~{\upshape\ref{thm:pig-general}} hold and also
\begin{newlist}

 \item[\quad {\normalfont (3)(ii)}$'$] 
$R$ contains a reflexive, antisymmetric binary relation $\sqsubseteq$.

\end{newlist}

\item[\normalfont (coD)]
Assume  that $\MT$ has named constants.
Assume that $\MT$ has a structural reduct $\MT^\flat$ in $\CY$ and let $\Omega$ be a subset of $ \CY(\MT^\flat, \NT)$.
Assume that $\N =\langle N; G^\nu, R^\nu\rangle$.
Then $\MT$ co-dualises $\M$ if
Conditions {\upshape (1), (2)} and {\upshape (3)(i)} in either Theorem~{\upshape\ref{thm:copig-simple}}  \textup(with $\Omega= \{\omega\}$\textup) or of Theorem~{\upshape\ref{thm:copig-general}} hold and  also 
\begin{newlist}
\item[\quad \normalfont (3)(ii)$'$]
$R$ contains a reflexive, antisymmetric binary relation $\sqsubseteq$.
\end{newlist}
\end{Dlist}
\end{corollary}

When $\M$ has a plentiful supply of both unary term functions and endomorphisms, we can combine the single-carrier theorems~\ref{thm:pig-simple} and~\ref{thm:copig-simple} to obtain various piggyback strong duality theorems---we present
three such under a common umbrella. Parts (I) and (II) of Theorem~\ref{thm:pigstrong1} are typically applied when we are piggybacking on Priestley duality for distributive lattices with or without bounds.
Part (III) is typically applied when we are piggybacking on Hofmann--Mislove--Stralka duality for semilattices with or without bounds~\cite{HMS74} (see also~\cite[2.4 (p.~157)]{DW83}).

Note that, in the following two results, Lemma~\ref{lem:liftconstants} guarantees that we do not need to add any assumptions about named constants.

\begin{theorem}[Piggyback Strong Duality Theorem] \label{thm:pigstrong1}
The setting for this theorem is the Basic Assumptions.
Assume that $\M$ and $\MT$ have  structural reducts $\M^\flat$ in $\CB$ and  $\MT^\flat$ in~$\CY$, respectively, and let $\omega\in \CB(\M^\flat, \N)\cap \CY(\MT^\flat, \NT)$.

Consider the following three sets of additional conditions.
\begin{Ilist}
\item[\normalfont (I)]
\begin{enumerate}[\normalfont (1)]
\item
$\N$ is a total algebra, $\NT = \langle N;G^\nu,R^\nu,\Tp\rangle$, and
\item
\begin{iiilist}
\item[\normalfont (i)]
$\omega\circ \Clo_1(\MT)$ separates the points of $M$,

\item[\normalfont(ii)]
$\omega\circ \Clo_1(\M)$ separates the structure~$\MT^\flat$, and

\item[\normalfont(iii)] $\MT$ entails every relation in $\omegamax {\M} r$, for each
$r\in R^\nu$.
\end{iiilist}
\end{enumerate}
\item[\normalfont (II)]
\begin{enumerate}[\normalfont (1)]
\item
$\N =\langle N; G^\nu, R^\nu\rangle$, $\NT$ is a total algebra, and
\item
\begin{iiilist}
\item[\normalfont (i)]
$\omega\circ \Clo_1(\MT)$ separates the structure $\M^\flat$,

\item[\normalfont (ii)]
$\omega\circ \Clo_1(\M)$ separates the points
of~$M$,
and

\item[\normalfont (iii)] if $R^\nu
\ne \varnothing$, the operations in the type of $\MT$ are continuous and $\MT$ entails every relation in $\omegamax {\M} r$, for each $r\in R^\nu$.
\end{iiilist}
\end{enumerate}

\item[\normalfont (III)]
\begin{enumerate}[\normalfont (1)]
\item Both $\N$ and $\NT$ are total algebras.

\item
\begin{iiilist}
\item[\normalfont (i)]
$\omega\circ \Clo_1(\MT)$ separates the points of $M$, and

\item[\normalfont (ii)]
$\omega\circ \Clo_1(\M)$ separates the points of $M$.
\end{iiilist}
\end{enumerate}
\end{Ilist}
Assume that {\upshape(I)}, {\upshape(II)} or {\upshape(III)} applies.
Then
\begin{enumerate}[\qquad \ \ \normalfont(a)]

\item
$\MT$ fully dualises $\M$,

\item $\M$ is injective in $\CA$ and $\MT$ is injective in $\CX$, and consequently the duality is strong, and

\item for all $\A\in \CA$ and all $\X\in \CX$,
\[
\D\A^\flat \cong \mathrm H(\A^\flat) \ \text{  and } \ \E\X^\flat\cong
\mathrm K(\X^\flat)
\]
via the maps $\Phi_\omega^\A$ and $\Psi_\omega^\X$, respectively.

\end{enumerate}
\end{theorem}

The following result is an immediate consequence of the Piggyback Strong Duality Theorem.
In its statement, we need to swap a topology between two total structures with the same underlying set.  Consequently, we need to suspend
 briefly the notation $\MT$ for an alter ego of a structure~$\M$. Given a structure $\M = \langle M; G, H, R\rangle$ and a topology $\Tp$ on~$M$, we define $\M^\Tp \coloneqq  \langle M; G, H, R, \Tp\rangle$.
We will also have occasion to use $\M_\Tp$ as an alternative notation for~$\M^\Tp$.
Note that if $\M_1^\Tp$ and $\M_2^\Tp$ are topological structures
 with the same underlying set~$M$, then $\M_2^\Tp$ is an alter ego of $\M_1$ if and only if $\M_1^\Tp$ is an alter ego of~$\M_2$.

\begin{theorem}[Two-for-one Piggyback Strong Duality Theorem]
 \label{thm:two4one}
Let $\M_1^\Tp$, $\M_2^\Tp$, $\N_1^\Tp$ and $\N_2^\Tp$ be total topological structures
and assume that $\M_2^\Tp$ and $\N_2^\Tp$ are alter egos of $\M_1$ and $\N_1$, respectively. For $i\in \{1, 2\}$, define $\CA_i\coloneqq \ISP(\M_i)$, $\CB_i\coloneqq \ISP(\N_i)$, $\CX_i\coloneqq \IScP(\M_i^\Tp)$ and $\CY_i \coloneqq  \IScP(\N_i^\Tp)$.
Assume that $\N_1^\Tp$ fully dualises $\N_2$ with $\N_1^\Tp$ injective in $\CY_1$ and $\N_2$ injective in $\CB_2$, and that $\N_2^\Tp$ fully dualises $\N_1$ with $\N_2^\Tp$ injective in $\CY_2$ and $\N_1$ injective in $\CB_1$.
 Assume that $\M_1$ and $\M_2$ have structural reducts $\M_1^\flat$ and  $\M_2^\flat$ with $(\M_1^\flat)^\Tp$ in~$\CY_1$ and $(\M_2^\flat)^\Tp$ in~$\CY_2$. Let
 $\omega$ be a continuous map in $\CB_1(\M_1^\flat, \N_1)\cap \CB_2(\M_2^\flat, \N_2)$.

Consider the following two sets of additional conditions.

\begin{Ilist}
\item[\normalfont (I)]

\begin{enumerate}[\normalfont (1)]
\item
 $\N_1$ is a total algebra  and $\N_2 =\langle N; G^\nu, R^\nu\rangle$; 
 
\item
\begin{iiilist}

\item[\normalfont(i)]
$\omega\circ \Clo_1(\M_2)$ separates the points of $M$,

\item[\normalfont(ii)]
$\omega\circ \Clo_1(\M_1)$ separates the structure~$\M_2^\flat$, and

\item[\normalfont(iii)]
both
$\M_2$ and $\M_2^\Tp$ entail every relation in $\omegamax {\M_1} r$, for each 
$r\in R^\nu$.

\end{iiilist}
\end{enumerate}

\item[\normalfont (II)]
\begin{enumerate}[\normalfont (1)]
\item Both $\N_1$ and $\N_2$ are total algebras;  

\item
\begin{iiilist}
\item[\normalfont(i)]
$\omega\circ \Clo_1(\M_2)$ separates the points of $M$, and

\item[\normalfont(ii)]
$\omega\circ \Clo_1(\M_1)$ separates the points of $M$.
\end{iiilist}
\end{enumerate}
\end{Ilist}

Assume that {\upshape(I)} or {\upshape(II)} applies. Then
\begin{enumerate}[\qquad \ \ \normalfont(a)]
\item
$\M_2^\Tp$ fully dualises $\M_1$ with $\M_1$ injective in $\CB_1$ and $\M_2^\Tp$ injective in~$\CX_2$,

\item $\M_1^\Tp$ fully dualises $\M_2$ with $\M_2$ injective in $\CB_2$ and $\M_1^\Tp$ injective in~$\CX_1$,
and

\item the isomorphisms in the base categories given in Theorem~
{\upshape \ref{thm:pigstrong1}(c)}
 apply to both the duality between $\CA_1$ and $\CX_2$, with base categories $\CB_1$ and $\CY_2$, and the duality between $\CA_2$ and $\CX_1$, with base categories $\CB_2$ and~$\CY_1$.

\end{enumerate}

\end{theorem}

In practice, when applying the Two-for-one Piggyback Strong Duality Theorem, it suffices to know that $\N_2^\Tp$ strongly dualises $\N_1$. Indeed, if
$\N_1 = \langle N; G_1^\nu, R_1^\nu\rangle$  and $\N_2 = \langle N; G_2^\nu, R_2^\nu\rangle$ are finite compatible total structures with $R_1^\nu$ and $R_2^\nu$ finite and both $\N_1$ and $\N_2$ have named constants, then $\N_2^\Tp$ strongly dualises $\N_1$ if and only if $\N_1^\Tp$ strongly dualises $\N_2$, by~\cite[Theorem 6.9]{D06}.

\section{Applications to distributive-lattice-based algebras}\label{sec:DL}

In this section we consider piggybacking on Priestley duality, demonstrating in particular how the original piggyback duality theorems in \cite{DW85,DW83PB,DP87}
for algebras fit into the general framework for piggybacking developed in this paper.  Our co-duality results, by contrast, are wholly new.
We shall restrict attention to bounded distributive lattices.
Analogous theorems are available when one bound is omitted from the type, or both bounds are omitted.

As before, $\CCD$ will denote the variety of bounded distributive lattices.
This is generated, as quasivariety, by~$\Dalg$, the two-element algebra in $\CCD$.
We take $\N$ to be $\Dalg$ and its alter ego $\NT$ to be the discretely topologised two-element chain $\DT = \langle \{0, 1\}; \leq, \Tp\rangle$, so that $ \CY\coloneqq  \IScP(\DT)$ is  the category $\CP$ of Priestley spaces.
Here all the conditions demanded of $\N$, $\NT$, $\CB=\ISP(\N)$ and $\CY = \IScP(\NT)$ in the Basic Assumptions are satisfied.
 We say that a structure $\M$ is \emph{$\CCD$-based} if it has a structural reduct $\M^\flat$ in~$\CCD$.

The $\CCD$-based Piggyback Duality Theorem~\ref{thm:DLpig}, which is well known and has been used often in the case that $\M$ is a \emph{total algebra}, follows immediately from Corollary~\ref{cor:sqsubseteq}(D).
In the single carrier case $\Omega = \{\omega\}$ (and similarly in Theorem~\ref{thm:DLcopig-general}), Condition~(2) can be restricted to relations $r \nsubseteq \Delta_M$; recall Remark~\ref{rem:rnotinDeltaM}.

\begin{theorem}[$\CCD$-based Piggyback Duality Theorem] \label{thm:DLpig}
Let $\M$ be a $\CCD$-based total structure with structural reduct $\M^\flat$ in~$\CCD$.
 Then
an alter ego $\MT$ of~$\M$
 dualises $\M$ provided that there is a finite subset $\Omega$ of $ \CCD(\M^\flat, \Dalg )$ such that

\begin{enumerate}[\ \normalfont (1)]

\item[\normalfont(0)] each $\omega\in \Omega$ is continuous with respect to the topologies on $\MT$ and $\DT  $,

\item
$\Omega\circ \Clo_1(\MT)$ separates the points of~$M$, and

\item
 $\MT$ entails every relation in $\Omegamax {\M} \leq$. 
 
\end{enumerate}

\end{theorem}

The theorem tells us that every finite $\CCD$-based total structure $\M$ is dualisable: choose $\Omega = \CCD(\M^\flat, \Dalg )$, then $\MT \coloneqq  \langle M; \Omegamax {\M} \leq, \Tp\rangle$ dualises $\M$.
In fact, the NU Duality Theorem for total structures \cite[Theorem~4.10]{D06}
already tells us that $\M$ is dualised by $\MT' \coloneqq  \langle M; \mathbb S(\M^2), \Tp\rangle$, where $\mathbb S(\M^2)$ is the set of all non-empty substructures of $\M^2$. The advantage of the $\CCD$-based Piggyback Duality Theorem over the NU Duality Theorem is that in general the set $\Omegamax {\M} \leq$ is much smaller than the set $\mathbb S(\M^2)$.

 A slightly perplexing consequence of the Piggyback Duality Theorem in the
$\CCD$-based  case (and of the NU Duality Theorem) is that if $\M_1 = \langle M; G, R_1\rangle$ and $\M_2= \langle M; G, R_2\rangle$ are finite $\CCD$-based total structures with the same set $G$ of fundamental operations, then there is a single alter ego, namely the alter ego $\MT \coloneqq  \langle M; \Omegamax {\M} \leq, \Tp\rangle$, that dualises both $\M_1$ and $\M_2$.

We turn now to the co-duality theorem.
Here we piggyback on the Bana\-schewski duality between ordered sets and Boolean topological bounded distributive lattices~\cite{Ban76} (see also~\cite{DHP12}).

Let $\twoBf = \langle \{0, 1\}; \leq\rangle$ be the two-element chain and  $\twoT = \langle \{0, 1\}; \vee, \wedge, 0, 1, \Tp\rangle$  be the two-element bounded distributive lattice endowed with the discrete topology, so that $\CQ\coloneqq \ISP(\twoBf)$ and $\CCD_\Tp \coloneqq  \IScP(\twoT)$ are, respectively, the category of non-empty ordered sets---an extremely easy exercise---and the category of non-trivial Boolean-topological bounded distributive lattices---a non-trivial fact due to Numakura~\cite{Num57}, see also~\cite[Example 8.2]{CDFJ}. Banaschewski's duality theorem tells us that $\twoT$ fully dualises $\twoBf$. Moreover, $\twoBf$ is injective in $\CQ$ and $\twoT$ is injective in $\CCD_\Tp$.

The following result is an immediate consequence of the Piggyback Co-duality Theorem~\ref{thm:copig-general}.

\begin{theorem}[$\CCD$-based Piggyback Co-duality Theorem] \label{thm:DLcopig-general}
Let $\M = \langle M; G, R\rangle$ be a total structure with named constants,
let $\Tp$ be a Boolean topology on $M$ and assume that $\M_\Tp \coloneqq  \langle M; G, R, \Tp\rangle$ has a structural reduct $\M_\Tp^\flat$ in $\CCD_\Tp$ and that the operations in $G$ are continuous. Let $\M'$ be a total structure that is compatible with $\M_\Tp$.  Then $\M'$ dualises $\M_\Tp$ provided that there is a subset $\Omega$ of $ \CCD_\Tp(\M_\Tp^\flat, \twoT)$ such that

\begin{enumerate}[\ \normalfont (1)]

\item
$\Omega\circ \Clo_1(\M')$ separates the points of~$M$, and

\item
$\M'$ entails every relation in $\Omegamax \M \leq$.
\end{enumerate}

\end{theorem}

\begin{corollary}
Every Boolean-topological $\CCD$-based total structure $\M_\Tp$ is dual\-isable.
\end{corollary}

\begin{proof} Choose $\Omega = \CCD_\Tp(\M_\Tp^\flat, \twoT)$. Then $\M' \coloneqq  \langle M; \Omegamax \M \leq, \Tp\rangle$ dualises $\M_\Tp$ by Theorem~\ref{thm:DLcopig-general}.
\end{proof}

Our two final theorems in this section are immediate corollaries of the corresponding general results.

\begin{theorem}[$\CCD$-based Piggyback Strong Duality Theorem] \label{thm:DLfullpig1}
Let $\M$ be a
$\CCD$-based total structure with structural reduct $\M^\flat$ in~$\CCD$, let $\MT$ be an alter ego of $\M$ and define $\CA\coloneqq \ISP(\M)$ and $\CX\coloneqq \IScP(\MT)$.
Assume
there is an order relation $\leq$
that is conjunct-atomic definable from $\MT$ such that $\MT^\flat \coloneqq \langle M; \leq, \Tp\rangle$ is a Priestley space, and there exists $\omega\in \CCD(\M^\flat, \Dalg )\cap \CP(\MT^\flat, \DT  )$
such that
\begin{enumerate}[\ \normalfont(1)]

\item $\MT$ entails each binary relation $r$ on $M$
which forms a substructure of $\M^2$ that is maximal in $(\omega, \omega)^{-1}(\leq)$,

\item
$\omega\circ \Clo_1(\MT)$ separates the points of $M$, and

\item
$\omega\circ \Clo_1(\M)$ separates the structure~$\MT^\flat$, that is, if $a\nleqslant b$ in $\MT^\flat$, then there exists $t\in \Clo_1(\M)$ such that $\omega(t(a))= 1$ and $\omega(t(b))= 0$.
\end{enumerate}
Then
the
conclusions~\upshape{(a)--(c)} of Theorem~\upshape{\ref{thm:pigstrong1}} hold.
\end{theorem}

\begin{theorem}[$\CCD$-based Two-for-one Piggyback Strong Duality Theorem] \label{thm:DL241pig}
Let $\M_1$ and $\M_2$ be compatible total
structures and let $\Tp$ be a Boolean topology on $M$ such that both $\M_1^\Tp$ and $\M_2^\Tp$ are topological structures, and assume that

\begin{enumerate}[\normalfont (i)]

\item there are binary term functions $\vee$ and $\wedge$ and nullary term functions $0$ and $1$ on $\M_1$ such that $(\M_1^\flat)^\Tp \coloneqq  \langle M; \vee, \wedge, 0, 1, \Tp\rangle$ is a Boolean topological bounded distributive lattice, and

\item there is an order relation $\leq$ that is conjunct-atomic definable from $\M_2$ such that $(\M_2^\flat)^\Tp \coloneqq \langle M; \leq, \Tp\rangle$ is a Priestley space.

\end{enumerate}
Assume that there exists a map $\omega\in \CCD(\M_1^\flat, \Dalg )\cap \CP((\M_2^\flat)^\Tp, \DT)$ such that
\begin{enumerate}[\ \normalfont(1)]

\item
$\omega\circ \Clo_1(\M_2)$ separates the points of $M$,

\item
$\omega\circ \Clo_1(\M_1)$ separates the structure~$\M_2^\flat$, that is, if $a\nleqslant b$ in $\M_2^\flat$, then there exists $t\in \Clo_1(\M_1)$ such that $\omega(t(a))= 1$ and $\omega(t(b))= 0$, and

\item $\M_2$ and $\M_2^\Tp$ entail every binary relation $r$ on $M$
which forms a substructure of $\M^2$ that is maximal in $(\omega, \omega)^{-1}(\leq)$.

\end{enumerate}
Then
the
conclusions~\upshape{(a)--(c)} of Theorem~\upshape{\ref{thm:two4one}} hold.
\end{theorem}

Leaving aside
 strongness of the co-duality involved, which has
 not been recognised before,
Theorems~\ref{thm:DLfullpig1} and~\ref{thm:DL241pig}
 explain the observed behaviour of certain algebras, for example Stone algebras, double Stone algebras and De Morgan algebras, whose natural and Priestley duals `coincide'.
A discussion of this phenomenon of coincidence,  in the context of arbitrary finitely generated quasivarieties  of $\CCD$-based algebras,
is given by
Cabrer and Priestley
in~\cite[Section~2]{CPcop}; see in particular
\cite[Corollary~2.4]{CPcop}.
They  showed  more generally how, within the ambit
of  $\CCD$-based
piggybacking,  to pass from the natural dual space
$\D \A$
of an algebra~$\A$ to $\mathrm H(\A^\flat)$.   The motivation in \cite{CPcop},
in
the context of
a study of coproducts,  was to harness
simultaneously the categorical virtues  of a natural duality and the pictorial nature of Priestley duality.   This isolated illustration---and we could have provided
others---indicates that the usefulness of piggybacking
extends beyond the derivation of natural dualities, whether strong or not.

\subsection*{Applications to Ockham algebras}
An algebra $\A= \langle A; \vee, \wedge, \neg, 0, 1\rangle$ is an Ockham algebra if $\A^\flat \coloneqq  \langle A; \vee, \wedge, 0, 1\rangle$
belongs to~$\CCD$,
and $\neg$  is a dual endomorphism of $\A^\flat$, that is, $\neg 0 = 1$, $\neg 1 = 0$, and $\neg$ satisfies the De Morgan laws:
\[
\neg (a\vee b) = \neg a \wedge \neg b \text{ \ and \ } \neg (a\wedge b) = \neg a \vee \neg b.
\]
 The variety~$\CO$ of Ockham algebras has provided a valuable  example
for a number of developments in duality theory.  Below we shall show how Theorem~\ref{thm:DL241pig} can be applied to~$\CO$.  Along the way, we
recapture Davey and Werner's piggyback duality for~$\CO$, using essentially
their argument;
this duality was originally obtained, without piggybacking, by Goldberg \cite{G83}.

We first recall some well-known facts.
Now let $\gamma \colon \mathbb N_0\to \mathbb N_0$ be   the \emph{successor function}: $\gamma(n)\coloneqq  n+1$ and let $c$ denote Boolean complementation
on $\{ 0,1\}$.  Define $\M_1 \coloneqq  \langle \{0, 1\}^{\mathbb N_0} \mid  \vee, \wedge, \neg, \underline 0, \underline 1\rangle$, where
$\vee$ and $\wedge$ are defined pointwise, $\underline 0$ and $\underline 1$ are the constant maps onto $0$ and~$1$, respectively, and, for all $a\in \{0, 1\}^{\mathbb N_0}$ we have $\neg(a) \coloneqq  c\circ a\circ \gamma$.
Thus, $\neg$ is given by \emph{shift left and then negate}; for example,
$\neg(0110010 \dots) =  (001101 \dots)$.
 Here, and subsequently, we write elements of $\{0, 1\}^{\mathbb N_0}$
as binary strings.
Then $\M_1$ is an Ockham algebra.  Moreover, $\M_1$ is subdirectly irreducible and $\CO = \ISP(\M_1)$ (\cite{U79}; see also \cite{G83}).

The topology on $\{0,1\}^{\mathbb N_0}$ will be the product topology $\Tp$ coming from the discrete topology on $\{0, 1\}$. It is an easy exercise to see that the operations $\vee$, $\wedge$ and $\neg$ on $\M_1$ are continuous with respect to~$\Tp$. Hence $\M_1^\Tp$ is a topological algebra.
We now set up a structure $\M_2 = \langle  \{0,1\}^{\mathbb N_0}; u, \preccurlyeq\rangle$ that is compatible with~$\M_1$.
Let $u \colon \{0,1\}^{\mathbb N_0} \to \{0,1\}^{\mathbb N_0}  $ be the \emph{left shift operator}, given by $u(a) \coloneqq   a\circ \gamma$.  Thus,  for example,
$ u(0110010 \dots) =(110010 \dots)$.
Then $u \in \End(\M_1)$ and $u$ is clearly continuous. Define $\preccurlyeq$ to be the \emph{alternating order} on $\{0, 1\}^{\mathbb N_0}$, that is, for all $a, b\in \{0, 1\}^{\mathbb N_0}$,
\[
a \preccurlyeq b \iff a(0) \leq b(0) \And a(1) \geq b(1) \And a(2) \leq b(2) \And \cdots.
\]
Since $\preccurlyeq$ forms a subalgebra of $\M_1^2$, it follows that $\M_2$ is compatible with~$\M_1$.  It is an elementary exercise to show that $\preccurlyeq$ is closed in the product topology on $\{0,1\}^{\mathbb N_0}\times \{0,1\}^{\mathbb N_0}$, and so $\M_2^\Tp$ is a topological structure.
Since $(\M_1^\flat)^\Tp \coloneqq  \langle \{0,1\}^{\mathbb N_0}; \vee, \wedge, \underline 0, \underline 1, \Tp\rangle$ is a Boolean-topological
 bounded distributive lattice and $(\M_2^\flat)^\Tp \coloneqq  \langle  \{0,1\}^{\mathbb N_0}; \preccurlyeq, \Tp\rangle$ is a Priestley space, conditions
(i) and (ii) of Theorem~\ref{thm:DL241pig} are satisfied.

Define $\omega \coloneqq  \pi_0\colon \{0, 1\}^{\mathbb N_0}\to \{0, 1\}$.
Clearly, $\pi_0 \in \CCD(\M_1^\flat, \Dalg )\cap \CP((\M_2^\flat)^\Tp, \DT)$.
The set $\pi_0\circ \Clo_1(\M_2)$ separates the points of $M = \{0, 1\}^{\mathbb N_0}$: indeed, let $a,b \in \{0, 1\}^{\mathbb N_0}$ with $a \ne b$, then
\begin{align*}
a\ne b &\implies (\exists n\in \mathbb N_0) \, a(n) \ne b(n)\\
& \implies (\exists n\in \mathbb N_0) \, u^n(a) (0) =(a\circ \gamma^n)(0)  \ne (b\circ \gamma^n)(0) = u^n(b)(0)\\
&\implies (\exists n\in \mathbb N_0) \, (\pi_0\circ u^n)(a) \ne (\pi_0\circ u^n)(b).
\end{align*}
As $u$ is in $\Clo_1(\M_2)$, so is $u^n$. Hence $\pi_0\circ \Clo_1(\M_2)$ separates the points of~$M$, that is, Condition (1) of Theorem~\ref{thm:DL241pig} holds.

We now show that Condition (2) of Theorem~\ref{thm:DL241pig} holds.
We must prove that $\pi_0\circ \Clo_1(\M_1)$ separates the structure~$\M_2^\flat$, that is, if $a\not\preccurlyeq b$ in $\M_2^\flat$, then there exists $t\in \Clo_1(\M_1)$ such that $\pi_0(t(a))= 1$ and $\pi_0(t(b))= 0$. We have
\begin{align*}
a\not\preccurlyeq b \text{ in $\M_2^\flat$}
&\iff (\exists n\in \mathbb N_0) \begin{cases} a(n) =1 \And b(n) = 0, \ &\text{$n$ even}\\
a(n) = 0 \And b(n) = 1, \ &\text{$n$ odd}
\end{cases}\\
&\implies (\exists n\in \mathbb N_0)\ \neg^n(a)(0) = 1 \And \neg^n(b)(0) = 0\\
&\implies (\exists n\in \mathbb N_0)\ \pi_0(\neg^n(a)) = 1 \And \pi_0(\neg^n(b)) = 0
\end{align*}
as required, with $t(v) \coloneqq  \neg^n(v)$.

Finally, to establish Condition (3) of Theorem~\ref{thm:DL241pig}, we must find the binary relations $r$ on $M$ which form substructures of $\M_1^2$ that are maximal in $(\pi_0, \pi_0)^{-1}(\leq)$. We have
\[
(\pi_0, \pi_0)^{-1}(\leq) = \{\, (a, b)\in (\{0, 1\}^{\mathbb N_0})^2 \mid a(0) \leq b(0)\,\}.
\]
Let $r$ be a subalgebra of $\M_1^2$ with $r\subseteq (\pi_0, \pi_0)^{-1}(\leq)$. Then
\begin{align*}
(a, b)\in r
 \implies {}& (\forall n\in \mathbb N_0) \, (\neg^n(a), \neg^n(b))\in r\\
\implies {}& (\forall n\in \mathbb N_0) \, \neg^n(a)(0) \leq \neg^n(b)(0)\\
\implies {}& a(0) \leq b(0) \And a(1) \geq b(1) \And a(2) \leq b(2) \And \cdots\\
\iff {}& a \preccurlyeq b.
\end{align*}
Thus $r\subseteq {\preccurlyeq}$. Since $\preccurlyeq$ forms a subalgebra of $\M_1^2$ and ${\preccurlyeq} \subseteq (\pi_0, \pi_0)^{-1}(\leq)$, and $r$ is a maximal such relation, it follows that $r = {\preccurlyeq}$ is the unique such relation.
(For related results, see \cite[Section~3]{CPcop},  which applies  to  finitely generated $\CCD$-based quasi\-varieties and more particularly 
\cite[Lemma~3.5 and 3.6]{DP87}, concerning Ockham algebra quasivarieties.) Since $\preccurlyeq$ is part of the type of $\M_2$, it is completely trivial that $\M_2$ and $\M_2^\Tp$ entail~$\preccurlyeq$, whence Condition~(3) of Theorem~\ref{thm:DL241pig} holds.

We therefore have the  following theorem.

\begin{theorem}[Two-for-One Strong Duality Theorem for Ockham algebras] \label{ex:Ockham}
Let $\M_1$ and $\M_2^\Tp$ be as defined above.
Then conclusions~\upshape{(a)--(c)} of Theorem~\ref{thm:two4one} hold. In particular, $\M_2^\Tp\coloneqq \langle \{0, 1\}^{\mathbb N_0}; u, \preccurlyeq, \Tp\rangle$ strongly dualises $\M_1$ and $\M_1^\Tp\coloneqq \langle \{0, 1\}^{\mathbb N_0}; \vee, \wedge, \neg, \underline 0, \underline 1, \Tp\rangle$ strongly dualises $\M_2$.
\end{theorem}

Not surprisingly, our proof that $\M_2^\Tp$ dualises the algebra $\M_1$ is essentially the same as the original Davey--Werner piggyback-based
proof.
We note that Goldberg \cite[Corollary~9]{G83} proved fullness via a direct calculation that $\X \cong \DE \X$ for each $\X \in \IScP(\M_2^\Tp)$.
Earlier, Goldberg~\cite[Theorem~4.4]{G81} had proved that $\M_1$ is injective in $\CO$, whence the Injectivity Lemma (see~\cite[Lemma~3.2.10]{NDftWA}) tells us that Goldberg's duality is strong.
Our piggyback-based proof that the duality is strong is new.
Conclusion~(c) of Theorem~\ref{thm:two4one} tells us, in particular, that the natural and Priestley duals of each Ockham algebra $\A\in \ISP(\M_1)$ and each Ockham space $\X\in \IScP(\M_2^\Tp)$ coincide; more formally, $\D\A^\flat \cong \mathrm H(\A^\flat)$ and $\E\X^\flat\cong \mathrm K(\X^\flat)$. These isomorphisms were first proved by Goldberg~\cite[Theorem~8]{G83}.
The fact that we can swap the topology and conclude that the topological algebra $\M_1^\Tp$ strongly dualises the structure $\M_2$ is new.

\begin{remark}
We have focused on the variety
$\CO$ rather than on its finitely generated subquasivarieties for several reasons.  First of all, $\CO$ provides a good example
of the applicability of our machinery in the non-finitely generated setting.
Secondly,  natural dualities for  subquasivarieties of $\CO$
generated by  finite subdirectly irreducible algebras
(in particular those which are also varieties) have been very extensively studied, both as a tool for investigating Ockham algebras and perhaps more importantly
as test case examples during the evolution of natural duality theory; see
\cite[Chapter~4]{NDftWA}
and \cite{G83, DW85,DP87,DHP12}.
The  most interesting in the present context are those Ockham varieties to which
Theorems~\ref{thm:DLfullpig1} and~\ref{thm:DL241pig} apply
and we thereby obtain new information.
For these varieties both the duality,  and the co-duality obtained by topology-swapping,
are strong,  and the forgetful functors on both sides give isomorphisms.
Varieties coming under this umbrella include De Morgan algebras,
Stone algebras and $\mathrm{MS}$-algebras
(and more generally any variety generated by a finite subdirectly irreducible Ockham algebra in one of the three infinite classes described in~\cite[Lemma~3.9]{DP87}).
Kleene algebras are not covered by
Theorems~\ref{thm:DLfullpig1} and~\ref{thm:DL241pig}.
See the references cited above for further details of the varieties concerned and of their piggyback natural dualities.
\end{remark}

\section{Applications to semilattice-based algebras}\label{sec:SL}

We concentrate in this section on the variety $\CS$ of meet-semilattices with~$1$,
but note (see \cite{DJPT07} and Example~\ref{ex:entropic}) that simple modifications produce corresponding results for meet-semilattices, meet-semilattices with~$0$, and bounded meet-semilattices.

When producing a piggyback duality theorem based on an underlying meet-semilattice structure, we are forced to have semilattice operations in the clones of $\M$ and $\MT$; denote these by~$\wedge$ and~$\sqcap$, respectively.
Since $\MT$ is an alter ego of~$\M$, the operation $\sqcap$ must be a homomorphism from $\M^2$ to~$\M$; in particular, we must have $(a \wedge c) \sqcap (b\wedge d) = (a \sqcap b) \wedge (c \sqcap d)$, for all $a,b,c,d \in M$. But it follows easily from this that $\sqcap = \wedge$; indeed,
\[
a \wedge b = (a \wedge b) \sqcap (b\wedge a) = (a \sqcap b) \wedge (b \sqcap a) = a \sqcap b.
\]
Thus we shall assume that in the clone of $\M$ there is a meet operation, $\wedge$, that is continuous with respect to the topology on $\MT$ and is a homomorphism from $\M^2$ to $\M$, and we shall assume that~$\wedge$ is part of the structure on~$\MT$.

Let $\S$ be the two-element meet-semilattice with $1$, let $\ST$ be $\S$ with the discrete topology added,
and take $\N$ to be $\S$ and $\NT$ to be $\ST$. Then the base categories
$\CS\coloneqq \ISP(\S)$ and $\CY \coloneqq  \IScP(\ST)$ are, respectively, the category of meet-semilattices with~$1$---an extremely easy exercise---and the category of Boolean-topological meet-semilattices with~$1$---see~\cite[2.4 (SEP)]{DW83} for a straightforward proof, and~\cite[Example 2.6 and Theorem 4.3]{CDFJ} to see the result from a more general perspective.
Moreover,  the parts of the Basic Assumptions concerning the base categories
are satisfied.

We shall apply the piggyback duality and co-duality theorems in their single-carrier versions.
Our first result generalises Davey, Jackson, Pitkethly and Talukder's Semilattice Piggyback Duality Theorem~\cite[Theorem~7.1]{DJPT07}  from semilattice-based algebras to semilattice-based total structures. Note that the semilattice operations in $\M^\flat$ and $\MT^\flat$ must agree.

\begin{theorem}[$\CS$-based
 Piggyback Duality Theorem] \label{thm:SLpig}
Let $\M$ be a total structure with a structural reduct $\M^\flat$ in~$\CS$, let $\MT$ be a structure with a structural reduct $\MT^\flat$ in~$\CY$ and assume that $\MT$ is an alter ego of $\M$.
Then $\MT$ dualises $\M$ provided there exists $\omega \in \CY(\MT^\flat, \ST)$ such that

\begin{enumerate}[\ \normalfont (1)]

\item $
\omega\circ \Clo_1(\MT)$ separates the points of~$M$, and

\item
$\MT$ entails each binary relation $r$ on $M$
 which forms a substructure of $\M^2$ that is maximal in $\ker(\omega)$.

\end{enumerate}
\end{theorem}

\begin{proof}
Since $\M^\flat$ is a total algebra (in fact a meet-semilattice with $1$), the assumptions guarantee that Conditions (0), (1), (2)(ii), (3)(i) and (3)(ii)(a)
of the  Single-carrier Piggyback Duality Theorem~\ref{thm:pig-simple} hold
(with (3)(i) holding vacuously).
Hence $\MT$ dualises~$\M$.
\end{proof}

The $\CS$-based Piggyback Co-duality Theorem, which is new, follows in a similar way to its non-co counterpart. Note that Lemma~\ref{lem:liftconstants} ensures that we do not need to mention named constants in this theorem.
In fact, since the base category $\CY$ has $1$ as a nullary operation in its type
and $\MT^\flat$ belongs to $\CY$, it follows that $\MT$ has a nullary operation and so has named constants.

\begin{theorem}[$\CS$-based Piggyback Co-duality Theorem] \label{thm:SLcopig}
Let $\M$ be a total structure with a structural reduct $\M^\flat$ in~$\CS$, let $\MT$ be a structure with continuous operations and with a structural reduct $\MT^\flat$ in~$\CY$. Then an alter ego $\MT$ of~$\M$ co-dualises $\M$ provided there exists $\omega \in \CY(\MT^\flat, \ST)$ such that
\begin{enumerate}[\ \normalfont (1)]

\item $
\omega\circ \Clo_1(\M)$ separates the points of~$M$, and

\item
$\M$ entails each binary relation $r$ on $M$
which forms a substructure of $\MT^2$ that is maximal in $\ker(\omega)$.

\end{enumerate}
\end{theorem}


Like the corresponding strong duality theorem for $\CCD$-based total structures,
our semilattice-based strong duality theorem is new. It is an immediate consequence of the Piggyback Strong Duality Theorem~\ref{thm:pigstrong1}(III).
By simply adding the assumption that $\MT$ is a topological structure, we could upgrade the theorem to the $\CS$-based Two-for-one Piggyback Strong Duality Theorem.

\begin{theorem}[$\CS$-based Piggyback Strong Duality Theorem] \label{thm:SLstrongpig}
Let $\M$ be a total structure with a structural reduct $\M^\flat$ in~$\CS$, let $\MT$ be a total structure with a structural reduct $\MT^\flat$ in~$\CY$ and assume that $\MT$ is an alter ego of $\M$.
Assume that there exists $\omega\in \CY(\MT^\flat, \ST)$ such that

\begin{enumerate}[\ \normalfont(1)]

\item
$\omega\circ \Clo_1(\MT)$ separates the points of $M$, and

\item
$\omega\circ \Clo_1(\M)$ separates the points of $M$.
\end{enumerate}
Then conclusions \upshape{(a)--(c)} of Theorem~\upshape{\ref{thm:pigstrong1}}
hold.
\end{theorem}

By applying this theorem when $\MT$ is simply $\M$ with an appropriate compact topology added, we obtain sufficient conditions for a semilattice-based total structure to be self-dualising. The  result in the case in which $\M$ is a total algebra was proved in~\cite[Theorem~7.4]{DJPT07}.
Applications of the self-dualising version of the theorem, including examples in which $\M$ is infinite, may be found in~\cite[Section~8]{DJPT07}.

We now present a new example that is closely related to the infinite example studied in~\cite[Section~8]{DJPT07}. Let $\V = \langle \{0, 1\}^{\mathbb N_0}; \wedge, u, \underline 1\rangle$, where $\wedge$ is defined pointwise, $\underline 1$~is the constant map onto $1$, and $u$ is the \emph{left shift operator}, that is, for all $a\in~\{0, 1\}^{\mathbb N_0}$ we have $u(a) \coloneqq   a\circ \gamma$ where $\gamma \colon \mathbb N_0\to \mathbb N_0$ is the \emph{successor function}: $\gamma(n)\coloneqq  n+1$. Let $\Tp$ be the product topology on $\{0, 1\}^{\mathbb N_0}$.

\begin{theorem}\label{ex:SLstrongpig}
Let $\W$ be topologically closed subalgebra of $\V$. Then $\W$ is strongly self dualising, that is, $\WT \coloneqq  \langle W; \wedge, u, \underline 1, \Tp\rangle$ strongly dualises~$\V$, where $\Tp$ is the subspace topology. In particular, $\V$ itself and every finite subalgebra of $\V$ is self dualising.
\end{theorem}

\begin{proof} The fact that $\W$ be topologically closed subalgebra of $\V$ guarantees that $\WT$ is an alter ego of $\W$. The calculations given for the Ockham algebra $\M_1$ in the proof of Theorem~\ref{ex:Ockham} show that $\pi_0 \circ \Clo_1(\W)$ separates the points of~$W$. It follows at once from Theorem~\ref{thm:SLstrongpig} that $\WT$ strongly dualises~$\W$.
\end{proof}

\begin{example} \label{ex:SLstrongpig2}
Some interesting examples of topologically closed subalgebras of $\V$ are listed below---see Figure~\ref{fig:subalgM}.

\begin{enumerate}[ (1)]

\item
$\mathit{CF} \coloneqq  \{\, a\in V \mid a^{-1}(1) \text{ is cofinite}\,\}$ forms a closed subalgebra of~$\V$.

\item For $k\in \mathbb N$, define $a_k\colon \mathbb N_0 \to \{0, 1\}$ by $a_k(\ell) = 1 \iff \ell\geq k$. Then, for all $n\geq 3$, the set $C_n \coloneqq \{ a_k \mid 1\leq k\leq n-2\}\cup\{\underline 0, \underline 1\}$ forms a subalgebra of~$\V$ and $C_\infty \coloneqq  \{ a_k \mid k\in \mathbb N\}\cup\{\underline 0, \underline 1\}$ forms a closed subalgebra of $\V$.

\item Fix $n\in \mathbb N$ and define $b \colon \mathbb N_0 \to \{0, 1\}$ by $b(\ell) = 1 \iff \ell \equiv 0 \pmod n$. For $k\in \mathbb Z_n$, define $b_k \coloneqq  u^{k}(b)$. Then $M_n \coloneqq  \{\, b_k \mid k\in \mathbb Z_n\,\} \cup\{\underline 0, \underline 1\}$ forms a subalgebra of~$\V$.
\item Define $a \coloneqq  01010101\dotsc$, $b \coloneqq  10101010\dotsc$ and $c\coloneqq  10000\dots$. Then $N_5 \coloneqq \{\underline 0, a, b, c, \underline 1\}$ forms a subalgebra of~$\V$.
\end{enumerate}
\end{example}

\begin{figure}[ht]
\begin{center}
  \begin{tikzpicture}
    \begin{scope}[xshift=0cm]
      \node at (0,-1) {$\C_n$};
      \node[element] (0) at (0,0) {};
      \node[element] (an-2) at ($(0)+(0,0.75)$) {};
      \node[element] (an-3) at ($(an-2)+(0,0.75)$) {};
      \coordinate (an-4) at ($(an-3)+(0,0.75)$) {};
      \coordinate (a2) at ($(an-4)+(0,0.25)$) {};
      \node[element] (a1) at ($(a2)+(0,0.75)$) {};
      \node[element] (1) at ($(a1)+(0,0.75)$) {};
      \coordinate (dots) at ($0.5*(an-4)+0.5*(a2)$) {};
      \draw[order] (0) to (an-2);
      \draw[order] (an-2) to (an-3);
      \draw[order] (an-3) to (dots);
      \draw[order] (dots) to (a1);
      \draw[order] (a1) to (1);
      \draw[loopy] (0) to [out=290,in=250] (0);
      \draw[unary] (an-2) to [bend left] (an-3);
      \draw[unary] (an-3) to [bend left] (an-4);
      \draw[unary] (a2) to [bend left] (a1);
      \draw[unary] (a1) to [bend left] (1);
      \draw[loopy] (1) to [out=110,in=70] (1);
      \node at ($(0)+(0.2,0)$) {\makebox[0pt][l]{$0$}};
      \node at ($(an-3)+(0.2,0)$) {\makebox[0pt][l]{$a_{n-3}$}};
      \node at ($(an-2)+(0.2,0)$) {\makebox[0pt][l]{$a_{n-2}$}};
      \node at ($(a1)+(0.2,0)$) {\makebox[0pt][l]{$a_{1}$}};
      \node at ($(1)+(0.2,0)$) {\makebox[0pt][l]{$1$}};
      \node[shape=rectangle,fill=white,inner sep=10pt] at (dots) {};
      \node at (dots) {\raise5.5pt\hbox{$\vdots$}};
    \end{scope}
    \begin{scope}[xshift=2.4cm]
      \node at (0,-1) {$\C_\infty$};
      \node[element] (0) at (0,0) {};
      \coordinate (a4) at ($(0)+(0,1)$) {};
      \node[element] (a3) at ($(a4)+(0,0.75)$) {};
      \node[element] (a2) at ($(a3)+(0,0.75)$) {};
      \node[element] (a1) at ($(a2)+(0,0.75)$) {};
      \node[element] (1) at ($(a1)+(0,0.75)$) {};
      \coordinate (dots) at ($0.55*(0)+0.45*(a3)$) {};
      \draw[order] (dots) to (a4);
      \draw[order] (a4) to (a3);
      \draw[order] (a3) to (a2);
      \draw[order] (a2) to (a1);
      \draw[order] (a1) to (1);
      \draw[loopy] (0) to [out=290,in=250] (0);
      \draw[unary] (a4) to [bend left] (a3);
      \draw[unary] (a3) to [bend left] (a2);
      \draw[unary] (a2) to [bend left] (a1);
      \draw[unary] (a1) to [bend left] (1);
      \draw[loopy] (1) to [out=110,in=70] (1);
      \node at ($(0)+(0.2,0)$) {\makebox[0pt][l]{$0$}};
      \node at ($(a3)+(0.2,0)$) {\makebox[0pt][l]{$a_3$}};
      \node at ($(a2)+(0.2,0)$) {\makebox[0pt][l]{$a_2$}};
      \node at ($(a1)+(0.2,0)$) {\makebox[0pt][l]{$a_1$}};
      \node at ($(1)+(0.2,0)$) {\makebox[0pt][l]{$1$}};
      \node[shape=rectangle,fill=white,inner sep=10pt] at (dots) {};
      \node at (dots) {$\vdots$};
    \end{scope}
    \begin{scope}[xshift=5.8cm]
      \node at (0,-1) {$\M_n$};
      \node[element] (0) at (0,0) {};
      \node[element] (1) at ($(0)+(0,2.9)$) {};
      \node[element] (a) at ($(0)+(-1.45,1.45)$) {};
      \node[element] (b) at ($(a)+(0.75,0)$) {};
      \coordinate (c) at ($(b)+(0.75,0)$) {};
      \coordinate (d) at ($(c)+(0.75,0)$) {};
      \coordinate (dots) at ($0.5*(c)+0.5*(d)$) {};
      \node[element] (e) at ($(d)+(0.75,0)$) {};
      \draw[order] (0) to (a);
      \draw[order] (0) to (b);
      \draw[order] (0) to (e);
      \draw[order] (a) to (1);
      \draw[order] (b) to (1);
      \draw[order] (e) to (1);
      \draw[loopy] (0) to [out=290,in=250] (0);
      \draw[unary] (a) to [bend left] (b);
      \draw[unary] (b) to [bend left] (c);
      \draw[unary] (d) to [bend left] (e);
      \draw[unary] (e) to [bend angle=20, bend left] (a);
      \draw[loopy] (1) to [out=110,in=70] (1);
      \node at ($(0)+(0.2,0)$) {\makebox[0pt][l]{$0$}};
      \node at ($(a)+(-0.2,-0.1)$) {\makebox[0pt][r]{$b_0$}};
      \node at ($(b)+(0.2,-0.1)$) {\makebox[0pt][l]{$b_1$}};
      \node at ($(e)+(0.2,-0.1)$) {\makebox[0pt][l]{$b_{n-1}$}};
      \node at ($(1)+(0.2,0)$) {\makebox[0pt][l]{$1$}};
      \node at (dots) {$\cdots$};
    \end{scope}
    \begin{scope}[xshift=9.8cm]
      \node at (0,-1) {$\N_5$};
      \node[element] (0) at (0,0) {};
      \node[element] (c) at ($(0)+(135:1.2)$) {};
      \node[element] (b) at ($(c)+(90:1.2)$) {};
      \node[element] (1) at ($(b)+(45:1.2)$) {};
      \node[element] (a) at ($0.5*(0)+0.5*(1)+(1,0)$) {};
      \draw[order] (0) to (c);
      \draw[order] (c) to (b);
      \draw[order] (b) to (1);
      \draw[order] (0) to (a);
      \draw[order] (a) to (1);
      \draw[loopy] (0) to [out=290,in=250] (0);
      \draw[unary] (c) to [bend right] (0);
      \draw[unary,<->] (a) to [bend angle=5, bend right] (b);
      \draw[loopy] (1) to [out=110,in=70] (1);
      \node at ($(0)+(0.2,0)$) {\makebox[0pt][l]{$0$}};
      \node at ($(a)+(0.2,0)$) {\makebox[0pt][l]{$a$}};
      \node at ($(b)+(-0.2,0)$) {\makebox[0pt][r]{$b$}};
      \node at ($(c)+(-0.2,0)$) {\makebox[0pt][r]{$c$}};
      \node at ($(1)+(0.2,0)$) {\makebox[0pt][l]{$1$}};
    \end{scope}
  \end{tikzpicture}
\caption{Some closed subalgebras of 
$\V$} 
\label{fig:subalgM}
\end{center}
\end{figure}
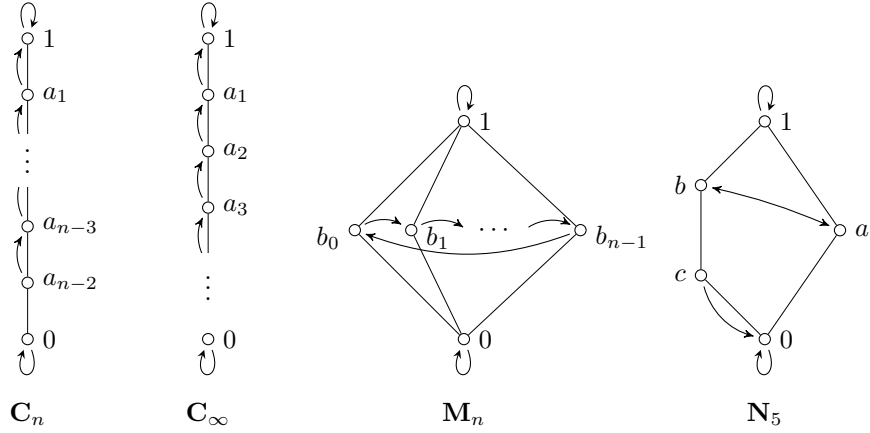

\begin{example}\label{ex:entropic}
The algebra $\mathbf E \coloneqq  \langle \{\underline 0, a_1, \underline 1\}; \wedge, u\rangle$, obtained by removing $\underline 1$ from the type of $\mathbf C_3$, is (isomorphic to) the entropic closure semilattice studied by Davey, Jackson, Pitkethly and Talukder~\cite{DJPT07}. By applying the variant of the
$\CS$-based
Piggyback Strong Duality Theorem~\ref{thm:SLstrongpig} obtained by piggybacking on the duality between semilattices and Boolean topological bounded semilattices, we can see immediately that $\twiddle E \coloneqq   \langle \{\underline 0, a_1,
\underline 1\};
\wedge, u, \underline 0, \underline 1, \Tp\rangle$ strongly dualises $\mathbf E$. This was proved directly in~\cite[Theorem~6.1]{DJPT07}.
\end{example}

\section{Proofs of the Piggyback Duality Theorems}\label{sec:proofs}

The proof of the Piggyback Duality Theorem~\ref{thm:pig-general} will be built from modularised components set out in a series of lemmas.  These components are combined to yield the proof of the theorem  which is  given after Lemma~\ref{lem:everything2}. 
The section concludes with the proof of the Strong Duality Theorem~\ref{thm:pigstrong1}, based on the Injectivity Lemma~\ref{lem:inj}.  

We shall begin  with a summary of our strategy for  proving the piggyback duality and co-duality theorems, in both their single carrier and general versions.
As noted already, we flag  duality and co-duality results with the tags (D) and (coD).

We now outline the roles of  our key lemmas, in their (D) versions.
For simplicity, assume for the purposes of this summary that we are working under the Basic Assumptions (not every lemma will require all the assumptions).  Our
objective  is to demonstrate that, under suitable assumptions, we can construct the
one-to-one map $d_{-}$ so that the diagram in Figure~\ref{fig:pig-fig1} commutes.
Assume, pro tem, that we have identified a candidate  alter ego~$\MT$ for $\M$ with
a structural reduct $\M^\flat$ in~$\CB$ and also that $\Omega$ is a selected subset of $\CB(\M^\flat,\N)$.

\begin{enumerate}[(1)]
\item
The  \textit{Commuting Triangle Lemma} \ref{lem:everything3} presupposes that  the map $d_{-}$ is well defined, for each $\A \in\CA$, and  that, for each $\alpha\in \ED \A$, the map $d_\alpha$ has domain $\CB(\A^\flat,\N)$,  that is,
\[
\bigcupPhi  = \CB(\A^\flat,\N).
\]
With these provisos, the lemma establishes that $d_{-}$ is one-to-one and that the diagram in Figure~\ref{fig:pig-fig2}  below commutes.

\item   
The \textit{Existence Lemma}~\ref{lem:everything4} addresses the issue of well-definedness of~$d_{-}$. Here the choice of alter ego $\MT$ comes into play,  and for the  first time we make use of entailment.

\item 
The \textit{Relation Preservation Lemma}~\ref{lem:everything5} and the \textit{Operation Preservation Lemma}~\ref{lem:pres3} combine to give various conditions sufficient, in combination with the conditions of the
{Existence Lemma}, to ensure that each map $d_\alpha$ is a $\CY$-morphism
(once we know that the family $\{\Phi_\omega^\A\}_{\omega\in\Omega}$ is jointly surjective, for all $\A \in \CA$).  The Operation Preservation Lemma is called on  only when $\NT$ is not purely relational. Topological input is provided by the \textit{Continuity Lemma}~\ref{lem:PhiContinuous} and, for the Operation Preservation Lemma, also  by the \textit{Morphism Lemma} \ref{lem:PhiMorphism}.  We comment on both these ancillary lemmas shortly.

\item 
Our goal is then to give conditions under which joint surjectivity holds.
The key steps towards this are provided by two important lemmas, the  \textit{Substructure  Lemma}~\ref{lem:everything1} and the \textit{Density Lemma}~\ref{lem:SepImpliesDense}.
In  the  former we need to impose restrictions on the structure~$\N$; and in the latter  we invoke the special properties of the duality for the base category included in the Basic Assumptions.

\item 
The \textit{Joint Surjectivity Lemma}~\ref{lem:everything2}  shows that, under conditions encountered already in (1)--(4), the maps in the set  $\{\Phi_\omega^\A\}_{\omega\in\Omega}$ are jointly surjective, as demanded in~(1) and~(3).
\end{enumerate}

We note that the multi-carrier case has certain features not present in the single-carrier case.
In the Substructure Lemma, and in the Joint Surjectivity Lemma which makes use of it, compatibility  issues arise  when  $|\Omega| > 1$  which force us to assume that  any operations in $\NT$ are unary or nullary.

The (coD) strand of the theory parallels the (D) strand quite closely,  the principal difference being in  the role of topology.
In almost every case  our lemmas have separate  (D) and (coD) parts.  In the proofs  we shall normally need to attend to  topological issues in just one of these (of course when $M$ and $N$ are finite no topological considerations arise) and then to prove one of the two structure assertions, calling on symmetry to obtain the other.
In such cases we shall simply  specify which of (D) and (coD) we elect to prove,
leaving it tacit  that the unproved  assertion also follows.
Exceptions occur with the Continuity Lemma and the Morphism Lemma, to which we alluded in~(3) above. The Continuity Lemma  has no (coD) component  and is
 required for the Piggyback Duality Theorem but not at all for  the Piggyback Co-duality Theorem. The Morphism Lemma does not have separated (D) and (coD) claims;
in it, the way in which a carrier map~$\omega$ relates to~$\CB$ and~$\CY$ simultaneously is crucial.
We issue a reminder that Lemma~\ref{lem:MaxImpliesClosed} comes into play
whenever we encounter entailment by $\M$. This applies to Part~(coD) in each
of the Existence Lemma, the $\sqsubseteq$-Lemma discussed below, and the Relation Preservation Lemma. 

The Existence Lemma involves entailment conditions linking the structures $\M$ and $\MT$ to the structures  $\N$ and $\NT$ via the carrier maps in $\Omega$.   The \textit{$\sqsubseteq$-Lemma}~\ref{lem:sublemma} shows that these conditions hold
in particular if there exists a reflexive, antisymmetric relation $\sqsubseteq $ on $N$ with suitable properties.  It is used to obtain Corollary~\ref{cor:sqsubseteq}.
The motivation here is to synthesise the behaviour of the order relation on~$\DT$ in the $\CCD$-based case.

The conditions imposed in Theorems~\ref{thm:pig-general} and~\ref{thm:copig-general}  guarantee that  the appropriate sets of carrier maps are jointly surjective.   It is nonetheless of  interest to ask whether this is an indispensable requirement if piggybacking is to be possible.
We are able to present  the \textit{Extended Commuting Triangle Lemma}~\ref{lem:everything3g} which does not demand joint surjectivity.

We are now ready to carry out our indicated programme.

\begin{lemma}[Commuting Triangle Lemma, for the jointly surjective case]
\label{lem:everything3}

Let $\M$ and $\N$ be structures and let $\MT$ and $\NT$ be alter egos of $\M$ and $\N$, respectively. Define $\CA\coloneqq \ISP(\M)$, $\CB \coloneqq \ISP(\N)$, $\CX\coloneqq \IScP(\MT)$ and $\CY \coloneqq  \IScP(\NT)$.

\begin{Dlist}
\item[\normalfont(D)] Assume that $\M$ has a structural reduct $\M^\flat$ in $\CB$ and let $\Omega$ be a subset of $ \CB(\M^\flat, \N)$.  Let $\A\in \CA$ and assume that for every morphism $\alpha\colon \CA(\A, \M) \to \MT$ there is a
map
\[
d_\alpha \colon \bigcupPhi \to N
\]
such that $d_\alpha \circ \Phi_\omega^\A = \omega \circ \alpha$, for all $\omega\in \Omega$
\textup(see Figure~\textup{\ref{fig:pig-fig2})}.

\begin{figure}[h!t]
\begin{center}
\begin{tikzpicture}
[auto,
 text depth=0.25ex,
 move up/.style=   {transform canvas={yshift=1.9pt}},
 move down/.style= {transform canvas={yshift=-1.9pt}},
 move left/.style= {transform canvas={xshift=-2.5pt}},
 move right/.style={transform canvas={xshift=2.5pt}}]
\node  (DA)  at   (-2,0)  {$\CA (\A, \M)$};
\node (M)   at  (2,0)   {$M$};
\node (Jim)  at( (-2,-2)   {$\phantom{\CA (\A, \M)}$};
 \node (N) at (2,-2)  {$N$};
 \node (Jim2) at  (-2.7,-1.9) {$\displaystyle \bigcupPhi$}; 
 \draw [-latex]  (DA) to node {$\alpha$} (M);
\draw [-latex] (DA) to node [swap] {$\Phi_\omega^{\A}$} (Jim);
\draw [-latex] (M) to node {$\omega$} (N);
\draw [-latex]  (Jim) to node [swap] {$d_\alpha$} (N);
\end{tikzpicture}

\caption{The equation $d_\alpha \circ \Phi_\omega^\A = \omega \circ \alpha$}\label{fig:pig-fig2}
\end{center}\end{figure}

\begin{enumerate}[\normalfont(1)]

\item If
$\Omega\circ \Clo_1(\MT)$ separates the points of~$M$, then the function $\alpha \mapsto d_\alpha$ is one-to-one.

\item
Assume that
the joint image of the family $\{\Phi_\omega^\A\}_{\omega\in \Omega}$
is $\CB(\A^\flat, \N)$.
Then $
d_{\esub {\scriptscriptstyle \A} (a)}  = k_{\A^\flat}(a)$, for all $a\in A$.
\end{enumerate}

\item[\normalfont(coD)] Assume that $\MT$ has a structural reduct $\MT^\flat$ in $\CY$ and let $\Omega$ be a subset of $ \CY(\MT^\flat, \NT)$.
Let $\X\in \CX$ and assume that for every morphism $u\colon \CX(\X, \MT) \to \M$ there is a map
\[
d_u \colon \bigcupPsi \to N
\]
such that $d_u \circ \Psi_\omega^\X = \omega \circ u$, for all $\omega\in \Omega$.

\begin{enumerate}[\normalfont(1)]

\item 
If $\Omega\circ \Clo_1(\M)$ separates the points of~$M$, then the function $u \mapsto d_u$ is one-to-one.

\item
Assume that the joint image of the family $\{\Psi_\omega^\X\}_{\omega\in \Omega}$ is $\CY(\X^\flat, \NT)$.
Then $d_{\epsub{\scriptscriptstyle \X}(x)}  = \kappa_{\X^\flat}(x)$, for all $x\in X$.

\end{enumerate}
\end{Dlist}
\end{lemma}

\begin{proof}  We prove the (D) assertions.  Those for (coD) are obtained likewise.

(D)(1) Let $\alpha, \beta \colon \CA(\A, \M) \to \MT$ be $\CX$-morphisms with $\alpha\ne \beta$. We shall show that $d_\alpha \ne d_\beta$. Since $\alpha\ne \beta$, there exists $x\in \CA(\A, \M)$ with $\alpha(x) \ne \beta(x)$. Thus, since $\Omega\circ \Clo_1(\MT)$ separates the points of $M$, there exists $g\in \Clo_1(\MT)$ and $\omega\in \Omega$ with $\omega(g(\alpha(x))) \ne \omega(g(\beta(x)))$. 
Since $g$ is a unary term function of~$\MT$, both $\alpha$ and $\beta$ preserve $g$, and $g\circ x \in \CA(\A, \M)$ as $g$ is an endomorphism of $\M$. 
Hence, since $d_\alpha \circ \Phi_\omega^\A = \omega \circ \alpha$ and $d_\beta \circ \Phi_\omega^\A = \omega \circ \beta$, we have
\begin{align*}
d_\alpha(\omega\circ g\circ x) &= d_\alpha(\Phi_\omega^\A(g\circ x)) = \omega(\alpha(g\circ x)) = \omega(\alpha(g^{\CA(\A, \M)}(x)))\\
& = \omega(g(\alpha(x))) \ne \omega(g(\beta(x))) \\
& = \omega(\beta(g^{\CA(\A, \M)}(x)))
 = \omega(\beta(g\circ x))  = d_\beta(\Phi_\omega^\A(g\circ x))\\
  &= d_\beta(\omega\circ g\circ x).
\end{align*}
Thus, $d_\alpha \ne d_\beta$, whence the map $\alpha \mapsto d_\alpha$ is one-to-one.

(D)(2)
Let $z \in \CB(\A^\flat, \N)$. By the assumed joint surjectivity, there  exist
 $\omega\in \Omega$ and $x\in \CA(\A, \M)$ with $z = \Phi_\omega^\A(x)$.
As $d_{\esub{\scriptscriptstyle \A} (a)}\circ \Phi_\omega^\A = \omega \circ \esubA (a)$, we have 
\begin{align*}
d_{\esub{\scriptscriptstyle \A} (a)}(z) = d_{\esub{\scriptscriptstyle \A} (a)}(\Phi_\omega^\A(x)) &= \omega(\esubA (a)(x))\\ &= \omega(x(a)) = k_{\A^\flat}(a)(\omega\circ x) = k_{\A^\flat}(a)(z).
\qedhere
\end{align*}
\end{proof}

\begin{lemma}[Existence Lemma]
\label{lem:everything4}
Let $\M$ and $\N$ be structures and let $\MT$ and $\NT$ be alter egos of $\M$ and $\N$, respectively.
Let
$\CA\coloneqq \ISP(\M)$, $\CB \coloneqq \ISP(\N)$, $\CX\coloneqq \IScP(\MT)$ and $\CY \coloneqq  \IScP(\NT)$.

\begin{Dlist}
\item[\normalfont(D)] Assume that $\M$ is a total structure that has a structural reduct $\M^\flat$ in $\CB$ and let $\Omega$ be a subset of $ \CB(\M^\flat, \N)$.  Let $\A\in \CA$ and let $\alpha \colon \CA(\A, \M) \to \MT$  be an $\CX$-morphism. Then, under either of the conditions listed below, there exists a  map
\[
d_\alpha \colon \bigcupPhi \to N
\]
such that $d_\alpha \circ \Phi_\omega^\A = \omega \circ \alpha$, for all $\omega\in \Omega$.

\begin{enumerate}[\normalfont (1)]

\item $\Omega = \{\omega\}$ and either
\begin{rmlist}
\item[\normalfont(i)]
$\MT$ entails each binary relation $r$ on $M$
which forms a substructure of $\M^2$ that is maximal in $\ker(\omega)$,
or

\item[\normalfont(ii)]
$\omega \circ \Clo_1(\M)$ separates the points of $M$ \textup(in which case $\Phi_\omega^\A$ is one-to-one\textup).
\end{rmlist}

\item The structure $\MT$ entails each relation in $\Omegamax {\M} {\Delta_N}$.

\end{enumerate}

\item[\normalfont(coD)] Assume that $\MT = \langle M; G, R, \Tp\rangle$ is a total structure that has a structural reduct $\MT^\flat$ in $\CY$ and let $\Omega$ be a subset of $\CY(\MT^\flat, \NT)$.  Let $\X\in \CX$ and let $u \colon \CX(\X, \MT) \to \M$ be an $\CA$-morphism.
Then, under either of the conditions listed below, there exists a map
\[
d_u \colon \bigcupPsi \to N
\]
such that $d_u \circ \Psi_\omega^\X = \omega \circ u$, for all $\omega\in \Omega$.

\begin{enumerate}[\normalfont(1)]
\item $\Omega = \{\omega\}$
and either

\begin{rmlist}
\item[\normalfont(i)]  
the operations in $G$ are continuous and $\M$ entails each binary relation $r$ on~$M$ which forms a substructure of $\MT^2$ that is maximal in $\ker(\omega)$,
or

\item[\normalfont(ii)]
$\omega \circ \Clo_1(\MT)$ separates the points of $M$ \textup(in which case $\Psi_\omega^\X$ is one-to-one\textup).
\end{rmlist}

\item 
The operations in $G$ are continuous and the structure $\M$ entails each relation in $\Omegamax {\stwiddle M} {\Delta_N}$.

\end{enumerate}

\end{Dlist}
\end{lemma}

\begin{proof}
We shall prove the (coD)  results.
First note that the condition  in (coD)(1)(i) for the single-carrier case is a
special instance of condition (coD)(2) since a binary relation on~$M$  forms a substructure of $\MT^2$ that is maximal in $\ker(\omega) = (\omega, \omega)^{-1}(\Delta_N)$ if and only if it belongs to $\omegamax {\stwiddle M} {\Delta_N}$.
 
Now  assume (coD)(2) holds. For $\gamma\in \Psi_\omega^\X( \CX(\X, \MT))$, we ``define" $d_u(\gamma) \coloneqq  \omega(u(\alpha))$, where~$\alpha$ is any element of $\CX(\X, \MT)$ for which $\Psi_\omega^\X(\alpha) =~\gamma$. Of course, we must show that this choice is independent of the choice of $\omega$ and $\alpha$. Assume that $\omega_1, \omega_2\in\Omega$. We must prove that $\omega_1(u(\alpha)) = \omega_2(u(\beta))$ whenever $\alpha, \beta\in \CX(\X, \MT)$ are such that  $\Psi_{\omega_1}^\X(\alpha) = \Psi_{\omega_2}^\X(\beta)$, that is,  $\omega_1\circ \alpha = \omega_2\circ \beta$. Since $\MT$ is a total structure, the image of the natural product morphism $\alpha\sqcap \beta \colon \X\to \MT^2$, that is, the relation $(\alpha\sqcap \beta)(X) = \{\, (\alpha(x), \beta(x)) \mid x\in X\,\}$, is a closed substructure of $\MT^2$, and since $\omega_1\circ \alpha = \omega_2\circ \beta$ it follows at once that $(\alpha\sqcap \beta)(X)\subseteq (\omega_1,\omega_2 )^{-1}(\Delta_N)$. Let $s$ be a substructure of $\MT^2$ that contains $(\alpha\sqcap \beta)(X)$ and is maximal in $(\omega_1,\omega_2 )^{-1}(\Delta_N)$; thus $s\in \Omegamax {\stwiddle M} {\Delta_N}$, and consequently $u$ preserves $s$, as $\M$ entails~$s$, by assumption. We have $(\alpha, \beta)\in s^{\CX(\X, \stwiddle M)}$, by construction, and hence $(u(\alpha), u(\beta))\in s$, as $u$ preserves~$s$. Since $s\subseteq (\omega_1,\omega_2 )^{-1}(\Delta_N)$, it follows that $\omega_1(u(\alpha)) = \omega_2(u(\beta))$, as required.

It remains to consider the case when  (coD)(1)(ii)
holds. To show that there exists a map $d_u\colon \Psi_\omega^\X( \CX(\X, \MT)) \to \N$ satisfying $d_u \circ \Psi_\omega^\X = \omega \circ u$, it suffices to prove that $\ker(\Psi _\omega^\X)\subseteq \ker (\omega\circ u)$.  In fact, we shall prove something much stronger, namely that $\ker(\Psi _\omega^\X) = \Delta_{\CX (\X, \stwiddle M)}$, that is, $\Psi _\omega^\X$ is one-to-one. Let $\alpha, \beta \in \CX(\X, \MT)$ with $(\alpha, \beta)\in \ker(\Psi _\omega^\X)$, so $\omega\circ \alpha = \omega \circ \beta$, and choose $x\in X$. Then, for all $t
\in \Clo_1(\MT)$, we have
{\allowdisplaybreaks
\begin{alignat*}{2}
\omega(t(\alpha(x))) &= \omega(\alpha(t(x)))&\hspace*{3em} &\text{as $\alpha$ preserves $t$}\\
& = \omega(\beta(t(x)))&
 &\text{as $\omega\circ \alpha = \omega \circ \beta$}\\
&= \omega(t(\beta(x)))& &\text{as $\beta$ preserves $t$}.
\end{alignat*}
}
Since $\omega \circ \Clo_1(\MT)$ separates the points of $M$, it follows that $\alpha(x) = \beta(x)$ and so
 $\alpha = \beta$, as $x\in X$ was chosen arbitrarily. Hence $\ker(\Psi _\omega^\X) = \Delta_{\CX (\X, \stwiddle M)}$.
\end{proof}


We shall now present the lemma  which gives sufficient conditions for
(D)(2) and (coD)(2) in the Existence Lemma to hold.  This result will immediately yield Corollary~\ref{cor:sqsubseteq} once the piggyback duality and co-duality theorems have been established.
Note that in the (coD) part of Corollary~\ref{cor:sqsubseteq}, we assume that
$\sqsubseteq$ is part of the structure on $\N$. As $\NT$ is an alter ego of $\N$, the relation $\sqsubseteq$ is therefore topologically closed.
 
\begin{lemma}[$\sqsubseteq$-Lemma] \label{lem:sublemma}
Let $\M$ and $\N$ be structures and let $\MT$ and $\NT$ be alter egos of~$\M$ and~$\N$ respectively.  Let $\CB\coloneqq \ISP(\N)$ and $\CY \coloneqq  \IScP(\NT)$.

\begin{Dlist}
\item[\normalfont (D)] 
Assume that $\M$ has a structural reduct $\M^\flat$ in  $\CB$ and let $\Omega$ be a subset of $\CB(\M^\flat,\N)$. Assume that $\sqsubseteq $ is a reflexive, antisymmetric binary relation on~$N$ such that $\MT$ entails each relation in $\Omegamax {\M} \sqsubseteq$. Then $\MT$ entails each relation in  $\Omegamax {\M} {\Delta_N}$.

\item[\normalfont (coD)]
Let $\MT = \langle M;G,R, \Tp\rangle$ with the operations in $G$ continuous. Assume that $\MT$ has a structural reduct $\MT^\flat$ in $\CY$ and let $\Omega$ be a subset of $\CY(\MT^\flat,\NT)$. Assume that $\sqsubseteq $ is a topologically closed reflexive, antisymmetric binary relation  on~$N$ such that $\M$ entails each relation in $\Omegamax {\stwiddle M} \sqsubseteq$. Then $\M$ entails each relation in  $\Omegamax {\stwiddle M} {\Delta_N}$.
\end{Dlist}
\end{lemma}

\begin{proof}
Topology plays no overt role in the proof here, so it would suffice to prove either statement, but to align with the proof of the Existence Lemma we establish  the (coD) statement. Under the assumptions made there we can also assert that
$\M$ entails each relation in $\Omegamax {\stwiddle M} \sqsupseteq$. We shall show that $\M$ entails every relation in $\Omegamax {\stwiddle M} {\Delta_N}$.

Let $s\in \Omegamax {\stwiddle M} {\Delta_N}$. Thus $s$ is a subuniverse of $\M^2$ and  there exist $\omega_1, \omega_2\in \Omega$ with $s$ maximal in $(\omega_1, \omega_2)^{-1}(\Delta_N)$. As $\sqsubseteq $ is reflexive, 
$\Delta_N$ is contained in ${\sqsubseteq }$ and so  $s\subseteq (\omega_1, \omega_2)^{-1}(\sqsubseteq)$ and $s\subseteq (\omega_1, \omega_2)^{-1}(\sqsupseteq )$. Let $t_1$ and $t_2$ be subuniverses of $\M^2$ with $s\subseteq t_1$ and $s\subseteq t_2$ such that $t_1$ and $t_2$ are maximal in $(\omega_1, \omega_2)^{-1}(\sqsubseteq)$ and $(\omega_1, \omega_2)^{-1}(\sqsupseteq)$, respectively; so $t_1\in \Omegamax {\stwiddle M} \sqsubseteq$ and $t_2\in \Omegamax {\stwiddle M} \sqsupseteq$. Let $(a, b)\in t_1\cap t_2$. Then $\omega_1(a)\sqsubseteq \omega_2(b)$ because $t_1\subseteq (\omega_1, \omega_2)^{-1}(\sqsubseteq)$, and $\omega_1(a)\sqsupseteq\omega_2(b)$ because $t_2\subseteq (\omega_1, \omega_2)^{-1}(\sqsupseteq )$. Hence $\omega_1(a) = \omega_2(b)$, as $\sqsubseteq $ is anti\-symmetric, and so $t_1\cap t_2\subseteq (\omega_1, \omega_2)^{-1}(\Delta_N)$. As $s\subseteq t_1\cap t_2$ and $t_1\cap t_2$ is a subuniverse, the maximality of $s$ in $(\omega_1, \omega_2)^{-1}(\Delta_N)$ guarantees that $s = t_1\cap t_2$. Since $\M$ entails $t_1$ and $t_2$, by assumption, it follows that $\M$ entails $s$.
\end{proof}

\begin{lemma}[Relation Preservation Lemma]
\label{lem:everything5} Let $\M$ and $\N$ be structures and let $\MT$ and $\NT$ be alter egos of $\M$ and $\N$, respectively. Define $\CA\coloneqq \ISP(\M)$, $\CB \coloneqq \ISP(\N)$, $\CX\coloneqq \IScP(\MT)$ and $\CY \coloneqq  \IScP(\NT)$.

\begin{Dlist}
\item[\normalfont(D)] 
Assume that $\M$ is a total structure that has a structural reduct $\M^\flat$ in $\CB$ and let $\Omega$ be a subset of $ \CB(\M^\flat, \N)$.  Let $\A\in \CA$ and let $\alpha \colon \CA(\A, \M) \to \MT$  be an $\CX$-morphism and assume that  there is a map
\[
d_\alpha \colon \bigcupPhi \to N
\]
such that $d_\alpha \circ \Phi_\omega^\A = \omega \circ \alpha$, for all $\omega\in \Omega$. Let $r$ be a subuniverse of~$\N^n$, for some $n\in \mathbb N$, and extend $r$ pointwise to the joint image of the family of maps $\{\Phi_\omega^\A\}_{\omega\in \Omega}$.
 Then $d_\alpha$ preserves $r$ provided the structure $\MT$ entails each relation in $\Omegamax {\M} r$.

\item[\normalfont (coD)] Assume that $\MT$ is a total structure with continuous operations that has a structural reduct $\MT^\flat$ in $\CY$ and let $\Omega$ be a subset of $\CY(\MT^\flat, \NT)$.  Let $\X\in \CX$ and let $u \colon \CX(\X, \MT) \to \M$  be an $\CA$-morphism and assume that there is a map
\[
d_u \colon \bigcupPsi \to N
\]
such that $d_u \circ \Psi_\omega^\X = \omega \circ u$, for all $\omega\in \Omega$. Let $r$ be a topologically closed subuniverse of $\N^n$, for some $n\in \mathbb N$,
and extend $r$ pointwise to the joint image of the family of maps $\{\Psi_\omega^\X\}_{\omega\in \Omega}$. Then $d_u$ preserves $r$ provided the structure $\M$ entails each relation in $\Omegamax {\stwiddle M} r$.

\end{Dlist}
\end{lemma}

\begin{proof}
We prove (D). The proof of (coD) is a simple modification.
Let $r$ be a subuniverse of $\N^n$ and assume that $\MT$ entails every relation in $\Omegamax {\M} r$. Let $z_1, \dots, z_n\in \bigcupPhi$  with $(z_1, \dots, z_n)\in r^{\CB(\A^\flat, \N)}$. Thus there exist $\omega_i\in\Omega$ and $x_i\in \CA(\A, \M)$ with $\Phi_{\omega_i}(x_i) = z_i$, that is $\omega_i\circ x_i = z_i$, for all $i\in \{1, \dots, n\}$. Since $d_\alpha(z_i) \coloneqq  \omega_i(\alpha(x_i))$, we must prove that $(\omega_1(\alpha(x_1)),\dots,  \omega_n(\alpha(x_n)))\in r$. As $\M$ is a total structure, the image of the natural product morphism $x_1\sqcap\dots \sqcap x_n \colon \A\to \M^n$, that is, the relation
\[
(x_1\sqcap\dots \sqcap x_n)(N)  = \{\, (x_1(a), \dots, x_n(a)) \mid a\in A\,\},
\]
is a substructure of $\M^n$. Since $(\omega_1\circ x_1,  \dots, \omega_n\circ x_n)\in r^{\CB(\A^\flat, \N)}$, we have $(\omega_1(x_1(a)),  \dots, \omega_n(x_n(a)))\in r$, for all $a\in A$, by the definition of $r^{\CB(\A^\flat, \N)}$. It follows at once that $(x_1\sqcap\dots \sqcap x_n)(N) \subseteq (\omega_1,\dots,\omega_n )^{-1}(r)$. Let $s$ be a substructure of $\M^n$ that contains $(x_1\sqcap\dots \sqcap x_n)(N) $ and is maximal in $(\omega_1,\dots,\omega_n )^{-1}(r)$; thus $s\in \Omegamax {\stwiddle M} r$. Then $(x_1,\dots, x_n)\in s^{\CA(\A, \M)}$ as $(x_1\sqcap\dots \sqcap x_n)(N) \subseteq s$, and hence $(\alpha(x_1), \dots, \alpha(x_n))\in s$, as $\alpha$ preserves $s$, by assumption. Since $s\subseteq (\omega_1,\dots,\omega_n )^{-1}(r)$, it follows that $(\omega_1(\alpha(x_1)),\dots,  \omega_n(\alpha(x_n)))\in r$.
\end{proof}


The Continuity Lemma concerns continuity of the maps $d_\alpha$. There is no (coD) assertion to be made about the maps $d_u$.

\begin{lemma}[Continuity Lemma]\label{lem:PhiContinuous}
Let $\M$ and $\N$ be structures and let $\MT$ and $\NT$ be alter egos of $\M$ and $\N$, respectively. Define $\CA\coloneqq \ISP(\M)$ and $\CB \coloneqq \ISP(\N)$. Assume that $\M$ has a structural reduct $\M^\flat$ in $\CB $.
\begin{enumerate}[\normalfont (1)]

\item Let $\omega\in \CB(\M^\flat, \N)$. Then $\Phi_\omega^\A \colon \CA(\A, \M) \to \CB(\A^\flat, \N)$ is continuous for all $\A\in \CA$ if and only if $\omega$ is continuous with respect to the topologies on $\MT$ and~$\NT$.

\item Let $\Omega$ be a finite subset of $ \CB(\M^\flat, \N)$
with each $\omega\in \Omega$ continuous.  Let $\A\in \CA$, let $\alpha \colon \CA(\A, \M) \to M$  be a continuous map and let
\[
d_\alpha \colon \bigcupPhi \to N
\]
be a map such that $d_\alpha \circ \Phi_\omega^\A = \omega \circ \alpha$, for all $\omega\in \Omega$. Then $d_\alpha$ is continuous.
\end{enumerate}
\end{lemma}

\begin{proof}
(1) Assume that $\omega$ is continuous and let $\A\in \CA$. Given sets $S$ and $T$, and  $s\in S$ and $U\subseteq T$, define $(s; U) \coloneqq  \{\, x\in T^S \mid x(s) \in U\,\}$. Thus a sub-basic open set in $\CB(\A^\flat, \N)$ is of the form $(a; U) \cap \CB(\A^\flat, \N)$, for some $a$ in $A$ and $U$ open in $N$.  We have
\begin{align*}
\bigl(\Phi_\omega^\A\bigr)^{-1}
((a; U) \cap \CB(\A^\flat, \N)) &= \{x\in \CA(\A,\M) \mid \Phi_\omega^\A(x)\in (a; U)\,\}\\
&= \{x\in \CA(\A,\M) \mid \omega(x(a))\in U\,\}\\
&= \{x\in \CA(\A,\M) \mid x(a)\in \omega^{-1}(U)\,\}\\
&= (a; \omega^{-1}(U)) \cap \CA(\A, \M),
\end{align*}
which is open in $\CA(\A, \M)$ as $\omega$ is continuous. Hence $\Phi_\omega^\A$ is continuous.

Now assume that $\Phi_\omega^\A$ is continuous for all $\A\in \CA$. In particular, $\Phi_\omega^{\mathbf F_1}$ is continuous, where $\mathbf F_1$ denotes the one-generated free algebra in $\CA$. Hence $\graph(\Phi_\omega^{\mathbf F_1})$ is closed in the product space $\CA(\mathbf F_1, \M) \times \CB(\mathbf F_1^\flat, \N)$. Let the free generator of $\mathbf F_1$ be $v_1$ and let $\pi_{v_1}^M \colon M^{F_1}
\to M$ and $\pi_{v_1}^N \colon N ^{F_1}\to N$ be the projections. Then
\begin{align*}
\graph(\omega) &= \{\, (a, \omega(a)) \mid a \in A\,\}
            = \{\, (x(v_1), \omega(x(v_1))\mid x\in \CA(\mathbf F_1, \M)\,\}\\
          &= (\pi_{v_1}^M\times \pi_{v_1}^N)(\graph(\Phi_\omega^{\mathbf F_1})).
\end{align*}
As the topologies on $M$ and $N$ are compact, the projection $\pi_{v_1}^M\times \pi_{v_1}^N$ is a closed map and consequently $\graph(\omega)$ is closed in $M\times N$. Since $\graph(\omega)$ is closed and the topologies on $M$ and $N$ are compact and Hausdorff, it follows that $\omega$ is continuous.

(2) By (1), $\Phi_\omega^\A$ is continuous and so is a closed map, for all $\omega\in\Omega$. Let $U$ be a closed subset of $D$. A simple calculation shows that
\[
d_\alpha^{-1}(U) = \bigcup_{\omega\in \Omega} \Phi_\omega^\A(\alpha^{-1}(\omega^{-1}(U))),
\]
which is closed in $\bigcup\{\, \Phi_\omega^\A( \CA(\A, \M))\mid \omega\in \Omega\,\}$ since $\alpha$ is continuous, each $\omega$ is continuous, each
$\Phi_\omega^\A$ is a closed map, and $\Omega$ is finite. Hence, $d_\alpha$ is continuous.
\end{proof}

Up to this point our lemmas have not reflected the dichotomy in Conditions~(2)
in our duality and co-duality theorems.
Now we address this. The Morphism Lemma~\ref{lem:PhiMorphism} and the Operation Preservation Lemma~\ref{lem:pres3} below will be  needed to handle two scenarios.
\begin{Dlist}
\item[(D)]  
In Theorem~\ref{thm:pig-general}, alternative (ii) under Assumption~(2) when $\NT$ is not purely relational,  and

\item[(coD)]  
In Theorem~\ref{thm:copig-general}, alternative (ii) under Assumption (2) when $\N$ is not purely relational,
\end{Dlist}
and likewise when
the corresponding conditions in the single-carrier versions of the theorems hold.

We have already noted that  the Morphism Lemma does not have separate (D) and (coD) claims.
In it we require a carrier map $\omega$ which acts both as a  $\CB$-morphism and as a $\CY$-morphism. In a (D) assertion, $\omega$ is usually a $\CB$-morphism, as this ensures that the associated map $\Phi_\omega^\A \colon \CA(\A,\M) \to \CB(\A^\flat, \N)$ is well defined. But in  order to assert that $\Phi_\omega^{\A} $ is a $\CY$-morphism we shall need $\omega $ to preserve the operations and relations not of $\N$ but of $\NT$.

Recall that, if $\M$ and $\MT$ have structural reducts $\M^\flat$ in $\CB$ and $\MT^\flat$ in $\CY$, respectively, then we have corresponding forgetful functors ${}^\flat \colon \CA \to \CB$ and ${}^\flat\colon \CX \to \CY$.

\begin{lemma}[Morphism Lemma]\label{lem:PhiMorphism}
Let $\M$ and $\N$ be structures and let $\MT$ and $\NT$ be alter egos of $\M$ and $\N$, respectively. Define $\CA\coloneqq \ISP(\M)$, $\CB\coloneqq \ISP(\N)$, $\CX\coloneqq \IScP(\MT)$ and $\CY \coloneqq  \IScP(\NT)$.
Assume that $\M$ and $\MT$ have structural reducts $\M^\flat$ in $\CB$ and $\MT^\flat$ in $\CY$, respectively.
Let $\omega \in\CB(\M^\flat, \N) \cap \CY(\MT^\flat, \NT)$.
Let $\A \in \CA$ and $\X \in \CX$. Then the map
$\Phi_\omega^\A \colon \CA(\A, \M)^\flat \to \CB(\A^\flat, \N)$
is a $\CY$-morphism and the map $\Psi_\omega^\X \colon \CX(\X, \MT)^\flat \to \CY(\X^\flat, \NT)$ is a $\CB$-morphism.
\end{lemma}

\begin{proof}
By symmetry, it is sufficient to prove the stronger statement, namely that $\Phi_\omega^\A$ is a $\CY$-morphism.
We note that $\Phi_\omega^\A$ is continuous, by Lemma~\ref{lem:PhiContinuous}. Since a map preserves an operation or partial operation $g$ if and only if it preserves $\graph(g)$, to prove that $\Phi_\omega^\A$ is a $\CY$-morphism it suffices to prove that $\Phi_\omega^\A$ preserves every relation preserved by $\omega$.
Let $r$ be an $n$-ary relation symbol (not necessarily in the type of $\MT$, nor of $\NT$), let $r^M$ and $r^N$ be interpretations of $r$ on $M$ and $N$, respectively, and assume that $\omega$ preserves~$r$.
Since $r^M$ and $r^N$ are extended pointwise to $\CA(\A, \M)$ and $\CB(\A^\flat, \N)$, respectively, we have, for all $x_1, \dots, x_n\in \CA(\A, \M)$,
\begin{align*}
(x_1, \dots, x_n)\in r^{\CA(\A, \M)} &\implies (\forall a\in A)\  (x_1(a), \dots, x_n(a))\in r^M\\
& \implies (\forall a \in A) \ (\omega(x_1(a)), \dots, \omega(x_n(a)))\in r^N\\
&  \implies (\omega\circ x_1, \dots, \omega\circ x_n)\in r^{\CB(\A^\flat, \N)} \\
&  \implies (\Phi_\omega^\A(x_1), \dots, \Phi_\omega^\A(x_n))\in r^{\CB(\A^\flat, \N)}.
\end{align*}
Hence $\Phi_\omega^\A$ preserves $r$.
\end{proof}

\begin{lemma}[Operation Preservation Lemma] \label{lem:pres3}
Let $\M$ and $\N$ be structures and let $\MT$ and $\NT$ be alter egos of $\M$ and $\N$, respectively. Define $\CA\coloneqq \ISP(\M)$, $\CB \coloneqq \ISP(\N)$, $\CX\coloneqq \IScP(\MT)$ and $\CY \coloneqq  \IScP(\NT)$. Assume that $\M$ and $\MT$ have  structural reducts $\M^\flat$ in $\CB$ and $\MT^\flat$ in $\CY$, respectively, and let $\Omega \subseteq \CB(\M^\flat,\N)\cap \CY(\MT^\flat,\NT)$.

\begin{Dlist}
\item[\normalfont(D)]
Assume that $\NT = \langle N; G^\nu, R^\nu, \Tp\rangle$ is a total structure, let $\A\in \CA$ and let $\alpha \colon \CA(\A, \M) \to \MT$  be an $\CX$-morphism.  Assume that the family of maps $\Phi_\omega^\A \colon \CA(\A, \M) \to \CB(\A^\flat, \N)$, for $\omega\in \Omega$,  is jointly surjective and that $d_\alpha \colon \CB(\A^\flat, \NT) \to N$ is a map such that $d_\alpha \circ \Phi_\omega^\A = \omega \circ \alpha$, for all $\omega \in \Omega$.
If $|\Omega| = 1$ or every operation in $G^\nu$ is unary or nullary, then $d_\alpha$ preserves the operations in $G^\nu$.

\item[\normalfont(coD)] 
Assume that $\N = \langle N; G^\nu, R^\nu\rangle$ is a total structure, let $\X\in \CX$ and let $u \colon \CX(\X, \MT) \to \M$  be an $\CA$-morphism.  Assume that
 the family of maps $\Psi_\omega^\X \colon \CX(\X, \MT) \to \CY(\X^\flat, \NT)$, for $\omega\in \Omega$, is jointly surjective and that $d_u \colon \CY(\X^\flat, \NT) \to N$  is a map such that $d_u\circ \Psi_\omega^\X = \omega \circ u$, for all $\omega\in \Omega$. If $|\Omega| = 1$ or every operation in $G^\nu$ is unary or nullary, then $d_u$ preserves the operations in $G^\nu$.\end{Dlist}
\end{lemma}

\begin{proof}
We shall prove (D). As $\MT^\flat$ is a structural reduct of $\MT$, the map $\alpha \colon \CA(\A, \M)^\flat \to \MT^\flat$  is a $\CY$-morphism and so preserves the operations in~$G^\nu$. Under the given assumptions Lemma~\ref{lem:PhiMorphism} applies, and
tells us that the maps $\Phi_\omega^\A$ are $\CY$-morphisms,  and so they also preserve the operations in~$G^\nu$. Fix $\omega\in \Omega$. As $\NT$ is a total structure, $Y_\omega \coloneqq  \Phi_\omega^\A(\CA(\A, \M))$ is a substructure of $\CB(\A^\flat, \N)$. Since $d_\alpha \circ \Phi_\omega^\A = \omega \circ \alpha$ and the maps $\Phi_\omega^\A$, $\alpha$~and $\omega$ preserve the operations in~$G^\nu$, it follows easily that $d_\alpha\rest{Y_\omega}\colon Y_\omega \to N$ preserves the operations in $G^\nu$---use the fact that if $\A$, $\B$ and $\C$ are algebras and $u\colon \A\to \B$ and $w\colon \A\to \C$ are  homomorphisms with $u$ surjective, then any map $v\colon B\to C$ satisfying $v \circ u = w$ is necessarily a homomorphism.

The joint-surjectivity assumption tells us that $\CB(\A^\flat, \N) = \bigcup_{\omega\in \Omega}Y_\omega$. If $\Omega = \{\omega\}$, then $\CB(\A^\flat, \N) = Y_\omega$, and hence $d_\alpha = d_\alpha\rest{Y_\omega}$ preserves the operations in~$G^\nu$. If every operation in $G^\nu$ is unary or nullary, then $d_\alpha$ preserves the operations in $G^\nu$ since its restriction to each of the subuniverses $Y_\omega$ does.
\end{proof}


We now work towards sufficient conditions for joint surjectivity.  We
prepare the way by establishing two lemmas concerning substructures.  They are
very different in the assumptions required to arrive at the desired conclusions. 

\begin{lemma}[Substructure Lemma] \label{lem:everything1}
Let $\M$ and $\N$ be structures and let $\MT$ and $\NT$ be alter egos of $\M$ and $\N$, respectively. Define $\CA\coloneqq \ISP(\M)$, $\CB\coloneqq \ISP(\N)$, $\CX\coloneqq \IScP(\MT)$ and $\CY \coloneqq  \IScP(\NT)$.

\begin{Dlist}
\item[\normalfont (D)]
Assume that $\M$ has a structural reduct $\M^\flat$ in $\CB$ and let $\Omega$ be a finite subset of $ \CB(\M^\flat, \N)$. Assume that one of the following conditions holds:
\begin{enumerate}

\item[\normalfont (i)]  $\NT =\langle N; R^\nu, \Tp\rangle$ is purely relational
and each $\omega\in\Omega$ is continuous;

\item[\normalfont  (ii)]  $\NT =\langle N; G^\nu, R^\nu, \Tp\rangle$
is a total structure, $\MT$ has a structural reduct $\MT^\flat$ in~$\CY$,
each $\omega\in \Omega$ belongs to $\CY(\MT^\flat, \NT)$, and
either $|\Omega| = 1$ or every operation in $G^\nu$ is unary or nullary.

\end{enumerate}
Then, for all $\A\in \CA$,  the joint image of the family $\{\Phi_\omega^\A\}_{\omega\in \Omega}$ is a closed substructure of $\NT^A$, and therefore of $\CB(\A^\flat, \N)$.

\item[\normalfont (coD)]
Assume that $\MT$ has a structural reduct $\MT^\flat$ in $\CY$ and let $\Omega$ be a subset of $ \CY(\MT^\flat, \NT)$ and assume  that one of the following conditions holds:
\begin{enumerate}
\item[\normalfont (i)] $\N =\langle N; R^\nu\rangle$ is purely relational;

\item[\normalfont (ii)]  $\N =\langle N; G^\nu, R^\nu\rangle$
is a total structure, $\M$ has a structural reduct $\M^\flat$ in~$\CB$, each $\omega\in \Omega$ belongs to $\CB(\M^\flat, \N)$, and either $|\Omega| = 1$ or every operation in $G^\nu$ is unary or nullary.
\end{enumerate}
 Then,  for all $\X\in \CX$, the joint image of the family $\{\Psi_\omega^\X\}_{\omega\in \Omega}$ is a substructure of $\N^X$, and therefore of $\CY(\X^\flat, \NT)$.
\end{Dlist}
\end{lemma}

\begin{proof}
Consider (D).
Define $Z$ to be the joint image of the family $\{\Phi_\omega^\A\}_{\omega\in\Omega}$. If  assumption (i) holds then it is trivial that $Z$ is a substructure of $\CB (\A^\flat, \N)$ as $G^\nu = \varnothing$. By Lemma~\ref{lem:PhiContinuous}, each $\Phi_\omega^\A$ is continuous since $\omega$ is. Thus $Z$ is topologically closed as~$\Omega$ is finite. Now assume that (ii) holds. By Lemma~\ref{lem:PhiMorphism}, each $\Phi_\omega^\A$ is a $\CY$-morphism. Hence, as $\NT$ is a total structure, the image of each $\Phi_\omega^\A$ is a topologically closed substructure of $\CB (\A^\flat, \N)$. If $\Omega = \{\omega\}$, then $Z$ is simply the image of $\Phi_\omega^\A$ and so is a closed substructure of $\CB (\A^\flat, \N)$. If  every operation in
 $G^\nu$ is unary or nullary, then $Z$ is a substructure since, in this case, a union of substructures is again a substructure. The set $Z$ is topologically closed since~$\Omega$ is finite and each $\Phi_\omega^\A$ is continuous. Hence (D)  follows. The proof of (coD) is almost identical, minus the topological arguments.
\end{proof}

The proof of the next  lemma, in which we need to assume that $\NT$ fully dualises~$\N$, is similar to the proof in the case that $\M$ is a total algebra given by Davey and Priestley~\cite[Proposition~1.11]{DP87}; we give the details for completeness.

\begin{lemma}[Density Lemma] \label{lem:SepImpliesDense}
Let $\N$ be a structure, let $\NT$ be a fully dualising alter ego of $\N$,
define $\CB = \ISP(\N)$ and $\CY \coloneqq  \IScP(\NT)$ and let $\langle \mathrm H, \mathrm K, k, \kappa\rangle$ be the associated dual equivalence between $\CB$ and~$\CY$. Let $\B\in \CB$ and $\Y\in \CY$.

\begin{Dlist}
\item[\normalfont (D)]
Assume that $\NT$ is injective in $\CY$. If $Z$ is a subset of  $\CB(\B, \N)$ that separates the structure $\B$, then the topologically closed substructure of $\CB(\B, \N)\lee \NT^B$ generated by $Z$ is $\CB(\B, \N)$ itself.

\item[\normalfont (coD)]
Assume that $\N$ is injective in $\CB$. If $C$ is a subset of  $\CY(\Y, \NT)$ that separates the structure $\Y$, then the substructure of $\CY(\Y, \NT)\le \N^Y$ generated by $C$ is $\CY(\Y, \NT)$ itself.
\end{Dlist}

\end{lemma}

\begin{proof} We shall prove (D). The proof of (coD)  is a simple modification obtained by replacing the closed substructure of $\mathrm H(\B)$ generated by $Z$ by the substructure of $\mathrm K(\Y)$ generated by $C$. Assume that $Z$ is a subset of $\mathrm H(\B)=\CB(\B, \N)$ that separates the structure $\B$. Define $\Y$ to be the closed substructure of $\mathrm H(\B)$ generated by $Z$. Let $\mu \colon \Y \to \mathrm H(\B)$ be the inclusion map and consider $\mathrm K(\mu) \colon \mathrm{KH}(\B) \to \mathrm K(\Y)$. We claim that $\mathrm K(\mu)$ is an embedding. Assume that $\N =\langle N; G^\nu, H^\nu, R^\nu\rangle$. Let $r$ be an $n$-ary relation in $\dom(H^\nu)\cup R^\nu\cup \{\Delta_N\}$ and let $\alpha_1, \dots, \alpha_n \in \mathrm{KH}(\B)$ with $(\alpha_1, \dots, \alpha_n) \notin r^{\mathrm{KH}(\B)}$. Since $\NT$ dualises~$\N$, there exist
$b_1, \dots, b_n \in B$ with $(b_1, \dots, b_n) \notin r^\B$ and $\alpha_i = k_\B(b_i)$, for all $i\in \{1, \dots, n\}$. Since $Z$ separates the structure $\B$, there exists $z\in Z\subseteq Y$ with $(z(b_1), \dots, z(b_n)) \notin r^{\N}$. Hence
\[
(\alpha_1(z), \dots, \alpha_n(z)) = (k_\B (b_1)(z), \dots, k_\B (b_n)(z)) = (z(b_1), \dots, z(b_n)) \notin r^{\N}.
\]
Thus $(\alpha_1\rest Y, \dots, \alpha_n\rest Y) \notin r^{\mathrm K(\Y)}$, and consequently $\mathrm K(\mu)$ is an embedding. As $\mu$ is an embedding and $\NT$ is injective in $\CY$, the map $\mathrm K(\mu)$ is
surjective.
Thus $\mathrm K(\mu)$ is an isomorphism. As $\NT$ fully dualises $\N$, it follows that $\mu$ is an isomorphism, whence $\Y = \mathrm H(\B)$.
\end{proof}

In Lemma~\ref{lem:everything2} we shall require $\Omega\circ \End(\M)$ to separate the structure $\M^\flat$ and $\Omega\circ \End(\MT)$ to separate the structure
$\MT^\flat$. In the earlier lemmas, by contrast, we  required that the possibly smaller sets $\Omega\circ \Clo_1(\MT)$ and $\Omega\circ \Clo_1(\M)$ separate  the points of~$M$.

\begin{lemma}[Joint Surjectivity Lemma]
\label{lem:everything2}
Let $\M$ and $\N$ be structures and let $\MT$ and $\NT$ be alter egos of $\M$ and $\N$, respectively. Define $\CA\coloneqq \ISP(\M)$, $\CB \coloneqq \ISP(\N)$, $\CX\coloneqq \IScP(\MT)$ and $\CY \coloneqq  \IScP(\NT)$.

\begin{Dlist}
\item[\normalfont (D)] Assume that $\M$ has a structural reduct $\M^\flat$ in $\CB$ and let $\Omega$ be a subset of $\CB(\M^\flat,\N)$.

\begin{enumerate}[\normalfont(1)]
\item
\begin{rmlist}
\item[\normalfont(i)] 
If the family $\{\Phi_\omega^\M\}_{\omega\in \Omega}$ is jointly surjective, then $\Omega\circ \End(\M)$ separates the structure $\M^\flat$.

\item[\normalfont(ii)] 
Assume that $\Omega\circ \End(\M)$ separates the structure~$\M^\flat$. Then, for each $\A\in \CA$, the joint image of the family $\{\Phi_\omega^\A \}_{\omega\in \Omega}$ separates the structure~$\A^\flat$.
\end{rmlist}

\item
Assume that $\Omega$ is finite, $\NT$ fully dualises $\N$  with $\NT$ injective in $\CY$, and  that one of the following conditions holds:
\begin{rmlist}
\item[\normalfont (i)]  
$\NT =\langle N; R^\nu, \Tp\rangle$
 is purely relational and each $\omega\in\Omega$ is continuous;

\item[\normalfont  (ii)]  
$\NT =\langle N; G^\nu, R^\nu, \Tp\rangle$ is a total structure, $\MT$ has a structural reduct $\MT^\flat$ in~$\CY$, each $\omega\in \Omega$ belongs to $\CY(\MT^\flat, \NT)$, and either $|\Omega| = 1$ or  every operation in $G^\nu$ is unary or nullary.
\end{rmlist}
 Assume that $\Omega \circ \End \M$ separates the structure $\M^\flat$.
Then, for each $\A \in \CA$, the family $\{\Phi_\omega^\A\}_{\omega\in \Omega}$
is jointly surjective.
\end{enumerate}

\item[\normalfont (coD)] Assume that $\MT$ has a structural reduct $\MT^\flat$ in $\CY$ and let $\Omega$ be a subset of $ \CY(\MT^\flat, \NT)$.

\begin{enumerate}[\normalfont(1)]

\item
\begin{rmlist}
\item[\normalfont(i)]
 If the family $\{\Psi_\omega^\stwiddle M\}_{\omega\in \Omega}$ is jointly surjective, then $\Omega\circ \End(\MT)$ separates the structure $\MT^\flat$.

\item[\normalfont(ii)] 
Assume that $\Omega\circ \End(\MT)$ separates the structure~$\MT^\flat$. Then, for each $\X\in \CX$, the joint image of the family $\{\Psi_\omega^\X\}_{\omega\in \Omega}$ separates the structure $\X^\flat$.
\end{rmlist}
\item
Assume that $\NT$ fully dualises $\N$ with $\N$ is injective in $\CB$.
Assume that the set $\Omega\circ \End(\MT)$ separates the structure~$\MT^\flat$, and
that one of the following conditions holds:
\begin{rmlist}
\item[\normalfont (i)] $\N =\langle N; R^\nu\rangle$ is purely relational;

\item[\normalfont (ii)]  $\N =\langle N; G^\nu, R^\nu\rangle$ is a total structure, $\M$ has a structural reduct $\M^\flat$ in~$\CB$, each $\omega\in \Omega$ belongs to $\CB(\M^\flat, \N)$, and either $|\Omega| = 1$ or every operation in $G^\nu$ is unary or nullary.
\end{rmlist}
Assume that $\Omega\circ \End(\MT)$ separates the structure~$\MT^\flat$. Then, for each $\X\in \CX$, the family $\{\Psi_\omega^\X\}_{\omega\in \Omega}$  is jointly surjective.
\end{enumerate}
\end{Dlist}
\end{lemma}

\begin{proof}
We prove (D). The proof of (coD) is almost identical, but the finiteness of $\Omega$ is not required since there is no topology involved.

(D)(1)(i) Assume that the family
$\{\Phi_\omega^\M\}_{\omega\in \Omega}$
is jointly surjective. Assume that $\N = \langle N; G^\nu, H^\nu, R^\nu\rangle$
and let $r\in \dom(H^\nu)\cup R^\nu\cup \{\Delta_N\}$ be $n$-ary. Suppose that $a_1, \dots, a_n\in M$ with $(a_1, \dots, a_n)\notin r^{\M^\flat}$. As $\M^\flat\in \ISP(\N)$, there exists a morphism $z\in \CB(\M^\flat, \N)$ with $(z(a_1), \dots, z(a_n))\notin r^{\N}$. By assumption, there exists $\omega\in\Omega$ and $x\in \CA(\M, \M) = \End(\M)$ with $\omega \circ x = \Phi_\omega^\A(x) = z$. Hence
$\Omega \circ \End(\M)$ separates the structure $\M^\flat$.

(D)(1)(ii) Assume that $\Omega \circ \End(\M)$ separates the structure~$\M^\flat$ and assume that $\N = \langle N; G^\nu, H^\nu, R^\nu\rangle$. Let $\A\in \CA$. Let $r$ be an $n$-ary relation in $\dom(H^\nu)\cup R^\nu \cup \{\Delta_N\}$ and let $a_1, \dots, a_n\in A$ with $(a_1, \dots, a_n)\notin r^{\A^\flat}$. Since $\A^\flat\in \ISP(\M^\flat)$, there exists $x\in \CB(\A^\flat, \M^\flat)$ with $(x(a_1), \dots, x(a_n))\notin r^{\M^\flat}$. As $\Omega \circ \End(\M)$ separates the structure~$\M^\flat$, there exists $\omega\in\Omega$ and $u\in \End(\M)$ satisfying $(\omega(u(x(a_1))), \dots, \omega(u(x(a_n))))\notin r^{\N}$. As $u\in \End(\M)$, it follows that  $u\circ x\in \CA(\A, \M)$ and hence $(\Phi_\omega^\A(u\circ x)(a_1), \dots, \Phi_\omega^\A(u\circ x)(a_n))\notin r^{\N}$. Therefore the joint image of
 the family $\{\Phi_\omega^\A\}_{\omega\in \Omega}$ separates the structure $\A^\flat$.

(D)(2) follows from (D)(1) and Lemmas~\ref{lem:everything1} and~\ref{lem:SepImpliesDense}.
\end{proof}

We  have now assembled all the components for the proof of the Piggyback Duality Theorem~\ref{thm:pig-general}. That of the Piggyback Co-Duality Theorem~\ref{thm:copig-general} is a simple modification.

\begin{proof}[\textbf{Proof of Theorem~{\upshape \ref{thm:pig-general}}}]
Assume that the Basic Assumptions are in force and that Conditions (0)--(3) in the statement of the theorem hold. For ease of reference we list these again here.
Recall that $\NT = \langle N;G^\nu,R^\nu,\Tp\rangle$. 
\begin{enumerate}[\normalfont (1)]

\item[\normalfont(0)]
$\Omega$ is finite and each $\omega\in \Omega$ is continuous with respect to the topologies on $\MT$ and $\NT$.

\item
$\Omega\circ \Clo_1(\MT)$ separates the structure~$\M^\flat$.

\item 
One of the following conditions holds:

\begin{rmlist}
\item[\normalfont(i)]  
$\NT$ is purely relational, or

\item[\normalfont(ii)]
$\MT$ has a structural reduct $\MT^\flat$ in~$\CY$, each $\omega\in \Omega$ belongs to $\CY(\MT^\flat, \NT)$, and every operation in $G^\nu$ is unary or nullary.
\end{rmlist}

\item
\begin{rmlist}
 \item[\normalfont (i)] 
The structure $\MT$ entails every relation in $\Omegamax {\M} r$, for each relation $r\in R^\nu$, and 

 \item[\normalfont (ii)]  
$\MT$ entails every relation in $\Omegamax {\M}{\Delta_N}$.
\end{rmlist}

\end{enumerate}

We now give the proof. Let $\A\in \CA$ and let $\alpha \colon \CA(\A, \M) \to \MT$ be an $\CX$-morphism, that is, $\alpha\in \ED\A$. We must prove that there exists $a\in A$ with $\esubA (a) = \alpha$. Since each $\omega\in \Omega$ is continuous, 
Conditions~(1) and~(2) ensure that the family $\{\Phi_{\omega}^\A\}_{\omega \in \Omega}$ is jointly surjective by the Joint Surjectivity Lemma~ \ref{lem:everything2}(D)(2). Condition (3)(ii) now ensures the existence of a map $d_\alpha \colon \CB(\A^\flat, \N) \to N$ satisfying $d_\alpha \circ \Phi_\omega^\A = \omega \circ \alpha$, for all $\omega\in \Omega$, by the Existence Lemma~\ref{lem:everything4}(D). Conditions (3)(i) and (0) allow us to invoke the
Relation Preservation Lemma~\ref{lem:everything5}(D) and the Continuity Lemma~\ref{lem:PhiContinuous}(2) to conclude that $d_\alpha$  preserves the relations in~$R^\nu$ and is continuous. If (2)(i) holds we deduce immediately that
 $d_\alpha$ is a $\CY $-morphism. To obtain the same conclusion when (2)(ii) holds we must confirm  also that $d_\alpha$ preserves the operations in $G^\nu$.
For this we can invoke the Operation Preservation Lemma~\ref{lem:pres3}(D).

Now we have a map $d_{-}\colon \ED\A \to \mathrm{KH}(\A^\flat)$, which, by the Commuting Triangle Lemma~\ref{lem:everything3}(2), satisfies
$d_{\esub{\scriptscriptstyle \A} (a)}  = k_{\A^\flat}(a)$, for all $a\in A$. Since $\NT$ dualises $\N$, the map $k_{\A^\flat} \colon \A^\flat \to \mathrm{KH}(\A^\flat)$ is surjective. So there exists $a\in A$ with $k_{\A^\flat}(a) = d_\alpha$. Hence $d_{\esub{\scriptscriptstyle \A} (a)} = k_{\A^\flat}(a) = d_\alpha$. 
Condition (1) guarantees that $d_{-} \colon \ED\A \to \mathrm{KH}(\A^\flat)$ is one-to-one, by Lemma~\ref{lem:everything3}(1). Hence~$\esubA (a) = \alpha$.

We have now established the version of the Piggyback Duality Theorem
which is not specialised to the single-carrier case.  For the proof for 
the single carrier case, that is, the proof of Theorem~\ref{thm:pig-simple}, we note that  essentially the same arguments as above apply when the conditions are stated in the simplified form which is possible when $|\Omega|=1$.
\end{proof}


We imposed sufficient conditions  above to guarantee joint surjectivity.  The following lemma indicates that we may, in suitable circumstances, be able to circumvent the need for joint surjectivity. This opens the way to the possibility of  pushing  through the piggyback constructions  in the absence of joint surjectivity,  subject to specified conditions being met---see~\cite[3.8]{DW85} for an application of the Extended Commuting Triangle Lemma to finite pseudocomplemented semilattices.
Since the examples considered in this paper are covered by the results stated in Section~\ref{sec:results}, we do not  pursue this topic further in this paper.

\begin{lemma}[Extended Commuting Triangle Lemma]
\label{lem:everything3g}

Let $\M$ and $\N$ be structures and let $\MT$ and $\NT$ be alter egos of $\M$ and $\N$, respectively. Define $\CA\coloneqq \ISP(\M)$, $\CB \coloneqq \ISP(\N)$, $\CX\coloneqq \IScP(\MT)$ and $\CY \coloneqq  \IScP(\NT)$.

\begin{Dlist}

\item[ \normalfont(D)] 
Assume that $\M$ has a structural reduct $\M^\flat$ in $\CB$ and let $\Omega$ be a subset of $ \CB(\M^\flat, \N)$.  Let $\A\in \CA$ and assume that for every morphism $\alpha\colon \CA(\A, \M) \to \MT$ there is a map $d_\alpha \colon \bigcupPhi \to N$
such that $d_\alpha \circ \Phi_\omega^\A = \omega \circ \alpha$, for all $\omega\in \Omega$ {\upshape(}as in Figure~{\upshape\ref{fig:pig-fig2}}{\upshape)}.

\begin{enumerate}[\normalfont(1)]

\item 
If $\Omega\circ \Clo_1(\MT)$ separates the points of~$M$, then the function $\alpha \mapsto d_\alpha$ is one-to-one.

\item
Assume that 
\begin{rmlist}
\item[\normalfont(i)]
the closed substructure of $\CB(\A^\flat, \N)$ generated by the joint image of the family $\{\Phi_\omega^\A\}_{\omega\in \Omega}$ is $\CB(\A^\flat, \N)$, and

\item[\normalfont(ii)]
$d_\alpha$ extends to a $\CY$-morphism $\widehat d_\alpha \colon \CB(\A^\flat, \N) \to \NT$.
\end{rmlist}
Then $\widehat d_{\esub {\scriptscriptstyle \A} (a)} = k_{\A^\flat}(a)$, for all $a\in A$.
\end{enumerate}

\item[\normalfont (coD)] Assume that $\MT$ has a structural reduct $\MT^\flat$ in $\CY$ and let $\Omega$ be a subset of $ \CY(\MT^\flat, \NT)$.
Let $\X\in \CX$ and assume that for every morphism $u\colon \CX(\X, \MT) \to \M$ there is a map $d_u \colon \bigcupPsi \to N$ such that $d_u \circ \Psi_\omega^\X
= \omega \circ u$, for all $\omega\in \Omega$.

\begin{enumerate}[\normalfont(1)]
\item 
If
$\Omega\circ \Clo_1(\M)$ separates the points of~$M$, then the function $u \mapsto d_u$ is one-to-one.

\item
Assume that 
\begin{rmlist}
\item[\normalfont (i)]
 the substructure of $\CY(\X^\flat, \NT)$ generated by the joint image of the family $\{\Psi_\omega^\X\}_{\omega\in \Omega}$ is $\CY(\X^\flat, \NT)$, and
 
\item[\normalfont (ii)]
 $d_u$ extends to a\/ $\CB$-morphism ${\widehat d_u \colon \CY(\X^\flat, \NT) \to \N}$.
\end{rmlist}
Then $\widehat d_{\epsub{\scriptscriptstyle \X}(x)}  = \kappa_{\X^\flat}(x)$, for all $x\in X$.
\end{enumerate}

\end{Dlist}
\end{lemma}

\begin{proof}
The differences between the Commuting Triangle Lemma~\ref{lem:everything3}  and this lemma lie in the weakened assumptions in (D)(2) and (coD)(2).
We consider (D)(2) only. We must prove that $\widehat d_{\esub{\scriptscriptstyle \A} (a)} = k_{\A^\flat}(a)$, for all $a\in A$. As $\widehat d_{\esub{\scriptscriptstyle \A} (a)}$ and $k_{\A^\flat}(a)$ are $\CY$-morphisms
from $\CB(\A^\flat, \N)$ to $\NT$, it suffices to prove that they agree on the generating set $Z\coloneqq  \bigcupPhi$. As $\widehat d_{\esub{\scriptscriptstyle \A} (a)}$ extends~$d_{\esub{\scriptscriptstyle \A} (a)}$, this is immediate  from the proof given for Lemma \ref{lem:everything3}(D)(2).
\end{proof}

We now turn to strong dualities. Our final lemma shows that when we impose sufficient conditions for $\MT$ to fully dualise $\M$, the resulting duality will in fact be strong.

\begin{lemma}[Injectivity Lemma] \label{lem:inj}
Let $\M$ and $\N$ be structures and let $\MT$ and $\NT$ be alter egos of $\M$ and $\N$, respectively. Let  $\CA\coloneqq \ISP(\M)$, $\CB \coloneqq \ISP(\N)$, $\CX\coloneqq \IScP(\MT)$ and $\CY \coloneqq  \IScP(\NT)$.

\begin{Dlist}
\item[\normalfont (D)] 
Assume that $\M$ has a structural reduct $\M^\flat$ in $\CB$, let $\omega\in \CB(\M^\flat, \N)$ and assume that $\Phi_\omega^\A \colon \CA(\A, \M) \to \CB(\A^\flat, \N)$ is a bijection for all $\A\in \CA$. If $\N$ is injective in $\CB$, then $\M$ is injective in $\CA$.

\item[\normalfont (coD)]  
Assume that $\MT$ has a structural reduct $\MT^\flat$ in $\CY$, let $\omega\in \CY(\MT^\flat, \NT)$ and assume that $\Psi_\omega^\X \colon \CX(\X, \MT) \to \CY(\X^\flat, \NT)$ is a bijection for all $\X\in \CX$. If $\NT$ is injective in $\CY$, then $\MT$ is injective in $\CX$.

\end{Dlist}
\end{lemma}

\begin{proof}
We shall prove only (coD). Let $\varphi\colon \X \to \MT^S$ be an embedding in~$\CX$, for some non-empty set~$S$, and let $u\colon \X\to \MT$ be an $\CX$-morphism. We must find an $\CX$-morphism $v\colon \MT^S \to \MT$ such that $u = v\circ \varphi$. As $\varphi^\flat \colon \X^\flat \to (\MT^S)^\flat = (\MT^\flat)^S$ is a $\CY$-embedding and $\omega \circ u^\flat \colon \X^\flat \to \NT$ is a $\CY$-morphism, the injectivity of $\NT$ in $\CY$ guarantees the existence of a $\CY$-morphism $\gamma \colon (\MT^\flat)^S \to \NT$ with $\omega \circ u = \gamma \circ \varphi$. As $\Psi_\omega^{\stwiddle M^S} \colon \CX(\MT^S, \MT) \to \CY((\MT^\flat)^S, \NT)$ is surjective, there exists an $\CX$-morphism $v\colon \MT^S \to \MT$ with $\omega \circ v = \gamma$. Since $u, v \circ \varphi\in \CX(\X, \MT)$, we can write
\[
\Psi_\omega^\X(u) = \omega \circ u = \gamma \circ \varphi = \omega \circ (v \circ \varphi) = \Psi_\omega^\X(v \circ \varphi).
\]
As $\Psi_\omega^\X$ is one-to-one, it follows that $u = v\circ \varphi$, as required.
\end{proof}

Now we present the proof of the  Piggyback Strong Duality Theorem~\ref{thm:pigstrong1}.

\begin{proof}[\textbf{Proof of Theorem~{\upshape \ref{thm:pigstrong1}}}]
(a)
Since (III) is a special case of both (I) and (II), it suffices to consider the assumptions in (I) and (II). Both of these sets of assumptions guarantee that Conditions (0), (1), (2)(ii) and (3)(ii)(b) of the Single-carrier Piggy\-back Duality Theorem~\ref{thm:pig-simple} hold, and that Conditions~(1), (2)(ii) and (3)(ii)(b) of the Single-carrier Piggyback Co-duality Theorem~\ref{thm:copig-simple} hold. Hence $\MT$ fully dualises~$\M$.

Part~(b) will follow immediately from the Injectivity Lemma~\ref{lem:inj} once we have proved Part (c). We shall prove (c) under the assumption that (I) holds.
By the Morphism Lemma~\ref{lem:PhiMorphism}, $\Psi_\omega^\X$ is a $\CY$-morphism, and by the Existence Lemma~\ref{lem:everything4}(coD)(1)(ii) and the Joint Surjectivity Lemma~\ref{lem:everything2}(coD)(2) (in which~(ii) is satisfied),
$\Psi_\omega^\X$ is one-to-one and onto. Since $\N$ is a total algebra, it follows that $\Psi_\omega^\X$ is an isomorphism between $\CX(\X,\MT)^\flat = \E\X^\flat$ and $\CY(\X^\flat, \NT) = \mathrm K(\X^\flat)$.

We can use the duality and the fact that $\E\X^\flat\cong \mathrm K(\X^\flat)$ to prove that $\D\A^\flat \cong \mathrm H(\A^\flat)$ as follows: since $\E\X^\flat\cong \mathrm K(\X^\flat)$, for all $\X\in \CX$, we have $\E{\D\A}^\flat \cong \mathrm K(\D\A^\flat)$, and so
\[
\D\A^\flat \cong \mathrm H(\mathrm K(\D\A^\flat)) \cong
\mathrm H(\E{\D\A}^\flat) \cong \mathrm H(\A^\flat).
\]
Alternatively, we can prove directly that $\Phi_\omega^\A \colon \CA(\A, \M)^\flat \to \CB(\A^\flat, \N)$ is an isomorphism in $\CY$ as follows.
By the Morphism Lemma~\ref{lem:PhiMorphism}, $\Phi_\omega^\A$ is a $\CY$-homomorphism, and by the Existence Lemma~\ref{lem:everything4}(D)(1)(ii)
and the Joint Surjectivity Lemma~\ref{lem:everything2}(D)(2) (in which~(ii)
is satisfied), $\Phi_\omega^\A$ is one-to-one and onto. It remains to show that, for each ($n$-ary) relation $r\in R^\nu$,
\[
(x_1, \dots, x_n)\notin r \text{ in } \CA(\A, \M)^\flat\implies(\Phi_\omega^\A(x_1), \dots, \Phi_\omega^\A(x_n))\notin r \text{ in } \CY(\A^\flat, \N).
\]
Let $x_1, \dots, x_n\in \CA(\A, \M)$ with $(x_1, \dots, x_n)\notin r$ in $\CA(\A, \M)^\flat$. Thus there exists $a\in A$ with $(x_1(a), \dots, x_n(a))\notin r$ in $\MT^\flat$. As $\omega\circ \Clo_1(\M)$ separates the structure~$\MT^\flat$, there exists $t\in \Clo_1(\M)$ with $\bigl(\omega(t(x_1(a))), \dots, \omega(t(x_n(a)))\bigr)\notin r$ in $\NT$.
Since $x_1, \dots, x_n$ preserve $t$ it follows that
\begin{multline*}
\bigl(\Phi_\omega^\A(x_1)(t(a)), \dots, \Phi_\omega^\A(x_n)(t(a))\bigr) = \bigl((\omega\circ x_1)(t(a)), \dots, (\omega\circ x_n)(t(a))\bigr)\\
= \bigl(\omega(x_1(t(a))), \dots, \omega(x_n(t(a)))\bigr)
= \bigl(\omega(t(x_1(a))), \dots, \omega(t(x_n(a)))\bigr)\notin r,
\end{multline*}
and so $(\Phi_\omega^\A(x_1), \dots, \Phi_\omega^\A(x_n))\notin r$ in $\CY(\A^\flat, \N)$. Hence (c) holds.
\end{proof}

\subsection*{Acknowledgement} The authors would like to thank
Leonardo  Cabrer for his careful reading of several versions of this paper. His well-targetted comments pinpointed some initial obscurities and have
also helped us  make  the end product much more user-friendly.


\end{document}